\documentclass[journal,letterpaper,twocolumn]{IEEEtran}
\usepackage{amsfonts}
\usepackage{amssymb}
\usepackage{cite}
\usepackage[cmex10]{amsmath}
\usepackage{float}
\usepackage{color}
\usepackage{stfloats,fancyhdr}

\usepackage{algorithm}
\usepackage{algorithmic}
\usepackage{multirow}
\usepackage{changepage}
\usepackage[normalem]{ulem}
\usepackage{flushend,cuted}

\newtheorem{lemma}{Lemma}
\newtheorem{proposition}{Proposition}

\newtheorem{remark}{Remark}

\IEEEoverridecommandlockouts

\ifCLASSINFOpdf
  \usepackage[pdftex]{graphicx}
  \DeclareGraphicsExtensions{.pdf,.jpeg,.png}
\else
  \usepackage[dvips]{graphicx}
  \DeclareGraphicsExtensions{.eps}
\fi

\usepackage{subfigure}

\usepackage{fancybox,dashbox}

\begin{document}

\title{Autonomous Demand Side Management based on Energy Consumption Scheduling and Instantaneous Load Billing: An Aggregative Game Approach
\thanks{This work of He (Henry) Chen was supported by International Postgraduate Research Scholarship (IPRS), Australian Postgraduate Award (APA), and Norman I Price Supplementary scholarship.}
\thanks{H. Chen, Y. Li, and B. Vucetic are with School of Electrical and Information Engineering, The University of Sydney, Sydney, NSW 2006, Australia (email: he.chen@sydney.edu.au, yonghui.li@sydney.edu.au, branka.vucetic@sydney.edu.au).}
\thanks{R. Louie is with the Electronic and Computer Engineering, Hong Kong University of Science and Technology, Clear Water Bay, Hong Kong (email:eeraylouie@ust.hk).}
}

\author{He~(Henry)~Chen,~\IEEEmembership{Student Member,~IEEE,} Yonghui Li,~\IEEEmembership{Senior Member,~IEEE,} \\
Raymond~H.~Y.~Louie,~\IEEEmembership{Member,~IEEE,} and Branka Vucetic,~\IEEEmembership{Fellow,~IEEE}
}

%\markboth{IEEE TRANSACTIONS ON XX
%%, VOL. XX, NO. XX, XXXX 2014
%}{ \MakeLowercase{\textit{et al.}}:  }

\maketitle

\begin{abstract}
In this paper, we investigate a practical demand side management scenario where the selfish consumers compete to minimize their individual energy cost through scheduling their future energy consumption profiles. We adopt an instantaneous load billing scheme to effectively convince the consumers to shift their peak-time consumption and to fairly charge the consumers for their energy consumption. For the considered DSM scenario, an aggregative game is first formulated to model the strategic behaviors of the selfish consumers. By resorting to the variational inequality theory, we analyze the conditions for the existence and uniqueness of the Nash equilibrium (NE) of the formulated game. Subsequently, for the scenario where there is a central unit calculating and sending the real-time aggregated load to all consumers, we develop a one timescale distributed iterative proximal-point algorithm with provable convergence to achieve the NE of the formulated game. Finally, considering the alternative situation where the central unit does not exist, but the consumers are connected and they would like to share their estimated information with others, we present a distributed synchronous agreement-based algorithm and a distributed asynchronous gossip-based algorithm, by which the consumers can achieve the NE of the formulated game through exchanging information with their immediate neighbors.
\end{abstract}


\IEEEpeerreviewmaketitle

\section{Introduction}
Recently, demand side management (DSM) has emerged as one of the key techniques to transform today's aging power grid into a more efficiently and more reliably operated smart grid  \cite{Ipakchi_PEM_2009,Fang_CST_2012}. Thanks to the two-way communication capabilities of smart grid, real-time pricing \cite{Samadi_SGC_2010} has been regarded as a promising technique to implement DSM due to its ability to effectively convince consumers to shift their peak-time energy consumption to non-peak time. In real-time pricing schemes, the energy price for a certain operation period is normally designed to be proportional to the aggregated load of all consumers during the considered period \cite{Samadi_SGC_2010,Mohsenian-Rad_TSG_2010,Baharlouei_TSG_2013,Atzeni_2013_TSG}. As a result, the consumers would prefer to consume more energy during non-peak times rather than peak times in order to decrease their energy cost. This can improve the operation efficiency of the whole grid since its demand is flattened.

In a real-time pricing based DSM framework, the billing mechanism (i.e., how to charge the consumers for their energy usage) is of great importance since it may significantly affect the consumers' motivation to participate in the DSM program. However, there has only been limited work investigating this important billing issue. \cite{Mohsenian-Rad_TSG_2010} proposed a simple billing approach, where the consumers were charged in proportional to their total energy consumption for the next operation period. This total load billing method can minimize the whole grid energy cost. However, the consumers are charged the same amount  if they consume the same total amount of electricity, regardless in peak or off-peak times, which leads to unfair charging for the consumers who use less electricity in peak times \cite{Baharlouei_TSG_2013}. To address this problem, \cite{Baharlouei_TSG_2013, Atzeni_2013_TSG} proposed a new billing approach, where each consumer is charged based on his/her instantaneous load in each time slot during the next operation period. As a result, the consumers will be charged more if they consume more during peak times and this can effectively improve the fairness of charging between different consumers \cite{Baharlouei_TSG_2013}. In this paper, the billing approach proposed in \cite{Baharlouei_TSG_2013, Atzeni_2013_TSG} is termed as instantaneous load billing, in contrast to the total load billing in \cite{Mohsenian-Rad_TSG_2010}. Based on the proposed billing approach, \cite{Atzeni_2013_TSG} also developed a classical non-cooperative game for the DSM scenario where the traditional consumers as well as consumers owing distributed energy sources and/or energy storage compete to reduce their energy bills. However, the main analysis and results in \cite{Baharlouei_TSG_2013,Atzeni_2013_TSG} are only valid when the energy price is a \emph{linear} function of the total load of all consumers in each time slot. Very recently, \cite{Atzeni_TSP_2013} extended \cite{Atzeni_2013_TSG} to the scenario with a \emph{general} energy price function. Based on the proximal decomposition method \cite{Scutari_Springer_2012}, synchronous and asynchronous algorithms were respectively developed in \cite{Atzeni_2013_TSG} and \cite{Atzeni_TSP_2013} for the consumers to achieve their optimal strategies in a distributed manner.

In this paper, we develop three novel distributed algorithms for autonomous DSM scenario, which enable the selfish consumers to optimize their own energy payment through scheduling their future energy consumption. The key contributions of this paper, with a particular emphasis on the differences with \cite{Atzeni_2013_TSG, Atzeni_TSP_2013}, are summarized as follows:

\textbf{(1)} Inspired by \cite{Baharlouei_TSG_2013,Atzeni_2013_TSG,Atzeni_TSP_2013}, we adopt the instantaneous load billing scheme to effectively convince the consumers to shift their peak-time energy consumption and fairly charge the consumers. In this paper, we are interested in a practical \emph{polynomial} energy price model instead of the general energy price model considered in \cite{Atzeni_TSP_2013}, since the polynomial model has been widely adopted in power systems (e.g., spot market price model \cite{Wang_conf_2010, Fan_2012_TSG}). By exploring the aggregative property of the instantaneous load billing scheme that the energy cost of each consumer only depends on its own and all consumers' aggregated energy consumption profiles \cite{Atzeni_2013_TSG, Atzeni_TSP_2013}, we develop a novel aggregative game\footnote{An aggregative game is a special kind of the non-cooperative game where each player's payoff is parameterized by its own action and the aggregative of the actions taken by all players \cite{Jensen_J_2010,Martimort_J_2012}\cite[Ch. 4]{Koshal_Thesis_2013}.} to model the strategic behaviors of the selfish consumers.
Additionally, we perform new theoretical analysis for the Nash equilibrium (NE) of the formulated game. This analysis will be facilitated by using advanced variational inequality theory \cite{Facchinei_VI_book}. As shown in this paper, the formulation of the aggregative game can facilitate the game analysis, the algorithm design and the convergence proof for the proposed algorithms. In our previous work \cite{Chen_ICC_2013}, a distributed and parallel gradient projection algorithm was proposed for the considered DSM framework. 

\textbf{(2)} For the algorithm design, we first consider the same setup as in \cite{Atzeni_2013_TSG,Atzeni_TSP_2013} where a central unit exists and broadcasts the real-time aggregated energy consumption profile to all consumers. In this case, the synchronous and asynchronous proximal decomposition algorithms proposed in \cite{Atzeni_2013_TSG,Atzeni_TSP_2013} can be directly applied to compute the NE of the formulated game. It should be noted that the algorithms in \cite{Atzeni_2013_TSG,Atzeni_TSP_2013} are \emph{two timescale}, which is due to the nature of the problem (i.e., the mapping function associated with distributed generation and storage is monotone) in \cite{Atzeni_2013_TSG,Atzeni_TSP_2013}. However, as shown later, the formulated problem in this paper can be guaranteed to possess strictly monotone mapping. Thus, we may not need to apply the two timescale algorithms, which are generally harder to implement in online settings than the \emph{one timescale} algorithms \cite{Kannan_CDC_2010}. Motivated by this, we develop a distributed iterative proximal-point algorithm to achieve the NE of the formulated game. This new algorithm is a parallel and one timescale algorithm and the choice of algorithm parameters does not depend on the system arguments.

\textbf{(3)} Considering the alternative situation without a central unit but where the consumers are connected and they exchange their estimated information with others, we develop a distributed agreement (consensus)-based algorithm, by which the consumers can achieve the NE of the formulated game through exchanging information with their immediate neighbors. Although information exchanges are required between the consumers in this algorithm, no private information (e.g., the exact energy consumption profile of each consumer) is shared between the consumers, thus effectively protecting the consumers' privacy. Moreover, the parameters of this algorithm can also be chosen without knowing the system's arguments a priori.

\textbf{(4)} Although the central unit is not necessary for the aforementioned agreement-based algorithm, synchronization between the consumers and coordination in terms of algorithm step sizes are still required, which are challenging in very large networks. Motivated by this, we develop a distributed asynchronous gossip-based algorithm for computing the NE of the formulated game without the need of a central unit. In this developed algorithm, synchronization is not required between the consumers. Besides the asynchronous updates, the consumers are allowed to use uncoordinated step sizes that are based on the frequency of the consumer update. Note that although the distributed consensus and gossip algorithms are well-known techniques, their application to achieve the NE of the formulated game is not straightforward at all and is not feasible without the formulation of the aggregative game in this paper.

{\bf \emph{Notations}}: All the vectors, except as specially stated, are column vectors. ${\bf{x}}^T$ and ${\left\| {\bf{x}} \right\|_2} = \sqrt {{{\bf{x}}^T}{\bf{x}}} $ denote the transpose and Euclidean norm of a vector $\bf{x}$, respectively. $A \times B$ is the cartesian product of sets $A$ and $B$. ${\bf{x}} = \left( {{{\bf{x}}_n}} \right)_{n = 1}^N $ denotes the operation of concatenating all vectors ${{{\bf{x}}_1}, \ldots ,{{\bf{x}}_N}} $ into a single column vector, i.e., ${\bf{x}}  = {\left( {{\bf{x}}_1^T, \ldots ,{\bf{x}}_N^T} \right)^T}$. To emphasize the $n$-th element within ${\bf{x}}$, we sometimes write $\left({\bf x}_n,{\bf x}_{-n}\right)$ instead of $\bf x$ with ${\bf x}_{-n} = \left( {{{\bf{x}}_m}} \right)_{m = 1,m \neq n}^N$. We use ${\left[ {\;\cdot\;} \right]_{\mathcal K}}$ to denote the Euclidean projection operator onto a set $\mathcal K$. ${\nabla _{\bf{x}}}f\left( {\bf{x}} \right)$ and $\nabla _{\bf{x}}^2f\left( {\bf{x}} \right)$ respectively denote the gradient vector and Hessian matrix of a scalar function $f\left( {\bf{x}} \right)$, while ${\bf{J}}_{\bf x}{\bf{F}}\left( {\bf{x}} \right)$ denotes the Jacobian matrix of a vector function ${\bf{F}}\left( {\bf{x}} \right)$.

The rest of this paper is organized as follows. The system model and the instantaneous load billing scheme are described in Section II. Section III formulates the aggregated game and analyzes the existence and uniqueness for the NE of the formulated game. The three new distributed algorithms are proposed in Section IV-VI, respectively. In section VII, numerical results are presented to illustrate and validate the theoretical analysis. Finally, Section VIII concludes this paper.

\section{System Model}

We consider an electricity network comprised of $N$ consumers, which are served by a common energy provider. We denote the set of these consumers as $\mathcal{N} = \{1,\ldots,N\}$. Each consumer is equipped with an energy management controller unit, which has full responsibility for scheduling the consumer's energy consumption. In addition, there exists a two-way communications network connecting each consumer to the energy provider. Similar to \cite{Mohsenian-Rad_TSG_2010,Samadi_TSG_2012}, we assume that the energy requirement of each consumer is determined in advance for $H$ future time slots. Each time slot can represent different timing horizons, e.g., one hour of a day.

\subsection{Energy Consumption Model}
We consider an energy consumption model as in \cite{Samadi_TSG_2012}, where the $n$th ($n\in {\mathcal N}$) consumer's energy consumption profile can be formulated as
\begin{equation}
{{\bf{q}}_n} = {\left( {q_n^1, \ldots ,q_n^H} \right)^T},
\end{equation}
where ${q_n^h}$ is the energy consumption of consumer $n$ in the $h$th time slot and it is subject to the following constraints:
\begin{equation}\label{const_1}
q_{n}^{h,\min} \le q_{n}^h \le q_{n}^{h,\max}\;{\rm{and}}\;
\sum\nolimits_{h = 1}^H {q_n^h = } {E_n},
\end{equation}
where $q_{n}^{h,\min}$ and $q_{n}^{h,\max}$ denote consumer $n$'s minimum and maximum energy levels\footnote{Note that the minimum and maximum energy levels can be estimated in practice by sophisticated predictive techniques, such as machine learning and stochastic signal processing \cite{Samadi_TSG_2012}. Moreover, the approaches presented in this work can be easily extended to the appliance-level energy consumption model \cite{Mohsenian-Rad_TSG_2010}.} in time slot $h$, respectively, and $E_n$ is the total energy requirement of consumer $n$ over all time slots. Therefore, the individual feasible energy consumption set of consumer $n$ can be expressed as
\begin{equation}\label{actionset}
\begin{split}
{\mathcal{Q}_n} = &\left\{ {{\bf{q}}_n}:\sum\nolimits_{h = 1}^H {q_n^h = } {E_n},\;{\rm and}\right.\\
&\;\;\;\left. q_n^{h,\min } \le q_n^h \le q_n^{h,\max },\;\forall h \in {\mathcal H} \right\},
\end{split}
\end{equation}
where ${\mathcal H} = \left\{ {1, \ldots ,H} \right\}$ is the set of all future $H$ time slots. The feasible energy consumption set of all consumers can thus be expressed as
\begin{equation}\label{eq:action_set_no_constraints}
 {\cal Q}  = { {\cal Q} _1} \times  \ldots  \times { {\cal Q} _N}.
\end{equation}

\subsection{Instantaneous Load Billing}
To effectively convince the consumers to shift their peak-time energy consumption and fairly charge the consumers for their energy consumption, we adopt the instantaneous load billing scheme \cite{Baharlouei_TSG_2013,Atzeni_2013_TSG,Atzeni_TSP_2013}, where the energy price (the cost of one unit energy) of a certain time slot is set as an increasing and smooth function of the total demand in that time slot, and the consumers are charged based on the instantaneous energy price as well as the energy amount they consume in each time slot. Instead of the general price model\cite{Atzeni_TSP_2013}, we focus on a practical and specific polynomial energy price model in this paper, which has been widely adopted in power systems (e.g., the spot market price model \cite{Wang_conf_2010, Fan_2012_TSG}). Specifically, the energy price of the $h$th time slot is given by:
\begin{equation}\label{f_price}
{p_h}\left( {{L_h}} \right) = a_h {\left( L_h \right)^{b_h}} + c_h,
\end{equation}
where $a_h,~b_h,~c_h$ are time slot-specific parameters with $a_h >0$, $b_h \ge 1$, $c_h \ge 0$, and $L_h$ is the total energy consumed by all consumers in time slot $h$.
It is should note that the price function in (\ref{f_price}) can readily account for the important characteristics of energy prices that are needed for DSM in smart grid. For example, the increasing and convex price function ensures that the energy price will grow more rapidly as the aggregated load increases. This can effectively convince the consumers to shift their peak-time consumption to non-peak hours, thereby flattening the overall demand curve and reducing the need for carbon-intensive and expensive peaking power plants. Therefore, the considered energy price model can improve the efficiency of the energy provider, and motivate and engage the energy provider to enforce such price model.

Follow the adopted energy price model, the total energy cost for consumer $n$ over all future $H$ time slots can thus be given~by:
\begin{equation}\label{eq:total_energy_cost}
{\mathcal B}_n\left( {{{\bf{q}}_n},{{\bf{q}}_ {-n} }} \right) = \sum\nolimits_{h = 1}^H {\left[ {{p_h}\left( { \sum\nolimits_{m=1}^N {q_m^h} } \right)q_n^h} \right]},
\end{equation}
where ${{\bf{q}}_{ - n}} = \left( {{{\bf{q}}_m}} \right)_{m = 1,m\neq n}^N $ denotes the $(N-1)H \times 1$ vector of all consumers' energy consumption profiles, except the $n$th one. This is in contrast to the total load billing method in \cite{Mohsenian-Rad_TSG_2010}, where the energy payment of the $n$th consumer is calculated~by
\begin{equation}\label{eq:total_load_billing}
{{\cal B}_n^{{\rm{TLB}}}}\left( {{{\bf{q}}_n},{{\bf{q}}_{ - n}}} \right) = \frac{{{E_n}}}{{\sum\nolimits_{m = 1}^N {{E_m}} }}\sum\limits_{h = 1}^H {\left[ {{p_h}\left( {\sum\limits_{m = 1}^N {q_m^h} } \right)\sum\limits_{m = 1}^N {q_m^h} } \right]}.
\end{equation}
It has been shown in \cite{Baharlouei_TSG_2013} that the adopted billing method in (\ref{eq:total_energy_cost}) is fairer than the total load billing method given in (\ref{eq:total_load_billing}). This will also be validated by the simulation results in this paper.

Note that (\ref{eq:total_energy_cost}) can be further rewritten as
\begin{align}\label{eq:total_energy_cost_1}
{{\cal B}_n}\left( {{{\bf{q}}_n},{{\bf{q}}_\Sigma }} \right) = \sum\nolimits_{h = 1}^H {\left[ {p_h\left( {q_\Sigma ^h} \right)q_n^h} \right]},
\end{align}
where ${{\bf{q}}_\Sigma } = \sum\nolimits_{m = 1}^N {{{\bf{q}}_m}}$ denotes the aggregated energy consumption profile of all consumers over future $H$ time slots and $q_\Sigma ^h = \sum\nolimits_{m = 1}^N {q_m^h}$ is the $h$th element of ${{\bf{q}}_\Sigma }$. From (\ref{eq:total_energy_cost_1}), we can see that the calculation of the total energy cost of each consumer only requires the knowledge of the aggregated energy consumption profile of all consumers (${{\bf{q}}_\Sigma }$), and that the individual consumption profile of each consumer (${{\bf{q}}_{-n} }$) is not required any more.

\section{Game Formulation and Analysis}\label{sec:game}
In this section, we formulate an aggregative game for the considered DSM scenario. By employing variational inequality theory, we then analyze the existence and uniqueness of the NE for the formulated aggregative game.
\subsection{Aggregative Game Formulation}
We consider the scenario where all consumers are selfish. In particular, each consumer aims to minimize his/her total cost through energy consumption scheduling. Mathematically, this will involve the $n$th consumer ($n \in {\mathcal N}$) solving the following optimization problem:
\begin{equation}\label{eq:opt_problem}
\begin{array}{*{20}{c}}
   ~{\mathop {\min }\limits_{{{\bf{q}}_n}} \;{{\mathcal B}_n}\left( {{{\bf{q}}_n},{{\bf{q}}_{ \Sigma}}} \right)} \\
   {{\rm{s}}{\rm{.}}\;{\rm{t}}{\rm{.}}\;{{\bf{q}}_n} \in  \mathcal{Q}_n }  \\
\end{array}.
\end{equation}

We can observe from (\ref{eq:opt_problem}) that the consumers solve optimization problems which are coupled with the aggregated energy consumption of all consumers. Hence, this energy consumption control scenario can be modeled by the following aggregative game \cite{Jensen_J_2010,Martimort_J_2012}\cite[Ch. 4]{Koshal_Thesis_2013}:
\begin{itemize}
\item \emph{{Players}}: The $N$ consumers.
\item \emph{{Actions}}: Each consumer selects its energy consumption ${\bf q}_n \in {\mathcal{Q}}_n$ to minimize his/her total energy cost.
\item \emph{{Payoffs}}: The total energy cost ${\mathcal{B}_n}\left( {{{\bf{q}}_n},{{\bf{q}}_{ \Sigma}}} \right)$ defined in (\ref{eq:total_energy_cost_1}).
\end{itemize}

For convenience, we denote this Nash equilibrium (NE) problem as $\mathcal{G} = \left\langle {\mathcal{N},\left\{ \mathcal{Q}_n  \right\},\left\{{\mathcal{B}_n}\left( {{{\bf{q}}_n},{{\bf{q}}_{ \Sigma}}} \right)\right\} }\right\rangle $. In the following subsection, we will employ variational inequality theory \cite{Facchinei_VI_book} to analyze the formulated game.

\subsection{NE Analysis}
Before proceeding, it is convenient to first present the following lemma regarding the properties of the formulated game's action sets and payoff functions:
\begin{lemma}\label{lemma:property_set_function}
For each $n = 1,\ldots,N $, the set ${\mathcal Q}_n \in \mathbb{R}^H$ is convex and compact, and each function ${\mathcal{B}_n}\left( {{{\bf{q}}_n},{{\bf{q}}_{ \Sigma}}} \right)$ is continuously differentiable in $ {{\bf{q}}_n}$. For each $n\in {\mathcal N}$ and each fixed tuple ${{\bf{q}}_{ - n}}$, the function ${{\cal B}_n}\left( {\cdot\;,\;\cdot\; + \sum\nolimits_{m = 1,m \ne n}^N {{{\bf{q}}_m}} } \right)$ is convex in ${\bf q}_n$ over the set ${\mathcal Q}_n$.
\end{lemma}
\begin{proof}
See Appendix \ref{appen:proof of lemma1}.
\end{proof}

Under Lemma \ref{lemma:property_set_function} and according to \cite[Prop. 1.4.2]{Facchinei_VI_book}, we have the following lemma:
\begin{lemma}\label{lemma:eq_game_VI}
The NE of the formulated game $\mathcal G$ is equivalent to the solution of the variational inequality (VI) problem\footnote{Given a subset ${\mathcal K}$ of the Euclidean $N$-dimensional space $\mathbb{R}^N$ and a mapping $\bf F$: ${\mathcal K}\rightarrow\mathbb{R}^N$, the variational inequality problem, denoted VI$\left({\mathcal K},{\bf F}\right)$, is to find a vector ${\bf x}^\ast \in {\mathcal K}$ such that ${\left( {{\bf{y}} - {{\bf{x}}^ * }} \right)^T}{\bf{F}}\left( {{{\bf{x}}^ * }} \right) \ge 0,\;\forall {\bf{y}} \in {\mathcal K}$.} denoted by VI$\left({\mathcal Q},{\bf F}\right)$ where ${\cal Q} = {{\cal Q}_1} \times  \ldots  \times {{\cal Q}_N}$ and
\begin{align}\label{eq:F_function}
{\bf{F}}\left( {\bf{q}} \right) = \left( {{{\bf{F}}_n}\left( {{{\bf{q}}_n},{{\bf{q}}_\Sigma }} \right)} \right)_{n = 1}^N,
\end{align}
with
\begin{equation}\label{eq:F_n_function}
{{\bf{F}}_n}\left( {{{\bf{q}}_n},{{\bf{q}}_\Sigma }} \right) = {\nabla _{{{\bf{q}}_n}}}{{\cal B}_n}\left( {{{\bf{q}}_n},{{\bf{q}}_\Sigma }} \right).
\end{equation}
\end{lemma}

By investigating the monotonicity property of the mapping ${\bf{F}}\left( {\bf{q}} \right)$, we can derive the following proposition:
\begin{proposition}\label{Prop:unique_NE}
If the price parameter ${b_h}$ satisfies ${b_h} < 3 + 4/\left( {N - 1} \right)$ for any $h \in {\mathcal H}$, then the formulated aggregative game admits a unique NE.
\end{proposition}
\begin{proof}
See Appendix \ref{appen:proof_of_prop1}.
\end{proof}

\begin{remark}
As can be seen from Proposition \ref{Prop:unique_NE},  only a specific relationship between the exponential factor of the polynomial price function and the number of consumers is required to guarantee the uniqueness of the NE. Specifically, the exponential factor of the price function is subject to an upper bound, which is inversely proportional to the number of consumers $N$. $~~~~~~~~~~~~~~~~~~~~~~~~~~~~~~~~~~~~~~~~~~~~~~~~~~~~~~~~~~~~~~\square$
\end{remark}

One could consider to solve the aforementioned game in a centralized manner, where a central unit adopts the algorithms proposed in \cite[Ch. 12]{Facchinei_VI_book} to solve the associated VI problem. However, such an approach requires each consumer to release detailed information about their energy consumption feasible set, which may lead to consumers' privacy and security concerns. To overcome this issue, in the following sections, we will develop three different distributed algorithms to achieve the NE of the formulated aggregative game for the scenarios with and without a central unit, which calculates the aggregated load and broadcasts it to consumers in each iteration of the algorithm.

\section{Distributed Iterative Proximal-point Algorithm with A Central Unit}
In this section, we consider the same setting as in \cite{Atzeni_2013_TSG,Atzeni_TSP_2013} where there is a central unit, which can provide the consumers with the latest information of the aggregated energy consumption profile after all consumers update their individual ones. In this case, we develop a distributed iterative Proximal-point algorithm to achieve the NE of the formulated aggregative game.

Before presenting our algorithm, it is worth mentioning that the formulated game can also be solved by the synchronous and asynchronous proximal decomposition algorithms proposed in \cite{Atzeni_2013_TSG} and \cite{Atzeni_TSP_2013}, which were guaranteed to converge under some conditions on the algorithm parameters. The distributed algorithms in \cite{Atzeni_2013_TSG,Atzeni_TSP_2013} were proposed based on the proximal decomposition method \cite{Scutari_Springer_2012} and solved a sequence of regulated versions of the original problem, each of which may need a distributed iterative process in itself. This is actually a two timescale approach (i.e., the proximal method updates at a slower timescale while solutions of the regularized problems change at a faster timescale) and is generally harder to implement in online settings \cite{Kannan_CDC_2010}. Additionally, the regulation parameter of such kind of algorithms has to be chosen centrally since it is normally dependent on the system arguments. It should be noted that the two-timescale property of the algorithms in\cite{Atzeni_2013_TSG,Atzeni_TSP_2013} is due to the nature of the problem (i.e., the mapping function associated with distributed generation and storage is monotone). However, as shown in Section \ref{sec:game}, the formulated problem in this paper can be guaranteed to possess strictly monotone mapping. Thus, we may not need to apply the two timescale algorithms. Motivated by this, we present a single timescale distributed algorithm based on the iterative regulation technique \cite{Kannan_CDC_2010,Kannan_SIAM_2012}, which requires only one projection step in each iteration. This algorithm is formally described in {\bf Algorithm 1}.

\begin{table}
\begin{center}
\begin{algorithm}[H] \label{Alg1}
\begin{algorithmic}[1]
\small
\STATE Set ${t} = 1$ and each consumer $n\in{\mathcal N}$ chooses a random ${\bf q}_n{\left(1\right)}$ from their feasible set ${\mathcal Q}_n$ and sends it to the central unit. The central unit calculates ${\bf{q}}_{\Sigma}{\left( {{1}} \right)} = \sum\nolimits_{n = 1}^N {{\bf{q}}_n{\left( {{1}} \right)}} $ and broadcasts it to the consumers. Given the values of the step-size $\gamma(t)$ and the parameter $\theta >0$.
\STATE If a suitable termination criterion is satisfied: $\rm{STOP}$.
\STATE For each consumer $n \in {\mathcal N}$:
\begin{adjustwidth}{0.5cm}{0cm}
3.1: Receive ${\bf{q}}_{\Sigma}{\left( {{t}} \right)} $ from the central unit.

3.2: Update the energy consumption profile by
\begin{equation*}
%\footnotesize
\begin{split}
{{\bf{q}}_n}\left( {t + 1} \right) =& \left[ {{{\bf{q}}_n}\left( t \right) - \gamma \left( t \right)\left( {{{\bf F}_n}\left( {{{\bf{q}}_n}\left( t \right),{\bf{q}}_{\Sigma}{\left( {{t}} \right)})} \right) + } \right.} \right.\\
&{\left. {\left. {  \theta \left( {{{\bf{q}}_n}\left( t \right) - {{\bf{q}}_n}\left( {t - 1} \right)} \right)} \right)} \right]_{{{\cal Q}_n}}}.
\end{split}
\end{equation*}
3.3: Send the update ${\bf{q}}_n{\left( {{t} + 1} \right)}$ to the central unit.
\end{adjustwidth}
\STATE
$t \leftarrow t + 1$; go to $\rm{STEP}$ 2.
\end{algorithmic}
\caption{: Distributed Iterative Proximal-point Algorithm
}
\end{algorithm}
\end{center}
\end{table}

The convergence property of Algorithm 1 is summarized in the following proposition:
\begin{proposition}\label{prop:convergence_alg_1}
Assume that the condition in Proposition \ref{Prop:unique_NE} holds. Then, the sequence of the energy consumption profile $\{{\bf {q}}(t)\}$ generated by Algorithm 1 converges to the unique NE of the game ${\mathcal G}$ if the step-size $\gamma(t)$ satisfies the following:
\begin{equation}\label{eq:alg1_stepsize_condition}
\sum\nolimits_{t = 1}^\infty  {\gamma \left( t \right) = \infty } \;{\rm{and}}\;\sum\nolimits_{t = 1}^\infty  {{\gamma ^2}\left( t \right) < \infty }.
\end{equation}
\end{proposition}
\begin{proof}
See Appendix \ref{appen:proof of prop2}.
\end{proof}

As shown above, Algorithm 1 can converge to the NE of the formulated game when there is a central unit that calculates the aggregated energy consumption profile ${\bf q}_{\Sigma}$ and broadcasts it to all consumers in each iteration. However, the developed Algorithm 1 and the algorithms in \cite{Atzeni_2013_TSG, Atzeni_TSP_2013} cannot be directly implemented for situations where the central unit does not exist, in which case the consumers thus do not have ready access to the aggregated energy consumption profile. Motivated by this issue, we will develop a distributed synchronous agreement-based algorithm and a distributed asynchronous gossip-based algorithm to achieve the NE of the formulated game in the following sections.

\section{Distributed Synchronous Agreement-based Algorithm without A Central Unit}\label{sec:alg_2}

In this section, we consider an alternative scenario where the central unit does not exist, but the consumers are connected in some manner and they are willing to share their estimated information through local communication. For this setting, we develop a distributed agrement-based algorithm, through which the consumers can achieve the NE of the game $\mathcal{G}$ via exchanging information with their immediate neighbors. In the developed algorithm, the connection topology of the consumers is modeled as an undirected (not necessarily complete) static graph. In practice, such a connection can be established through either wired or wireless communication techniques. Specifically, the connection can be implemented by employing the power line communication technique or using the resources of cellular networks to establish a virtual private network. As these techniques are widely deployed, the connection of a large number of consumers in large areas is feasible. Since only immediate connected consumers exchange information, the amount of data to be exchanged at each iteration of the developed algorithm is proportional to the numbers of connections between the consumers.

Recall that, in the formulated aggregative game, each consumer's payoff is only determined by his/her own energy consumption profile and the aggregated energy consumption profile of all consumers. Hence, the unique NE of the formulated game is achieved when the consumers reach an agreement (consensus) on the aggregated energy consumption profile. Following this equivalence and inspired by \cite[Ch. 4]{Koshal_Thesis_2013}, we develop a distributed agreement-based algorithm to achieve the unique NE of the considered aggregative game. In each iteration of this algorithm, each consumer $n \in {\mathcal N}$ executes the following three steps:
\begin{itemize}
  \item \textbf{\emph{Step 1}}: Estimate the average energy consumption of all consumers through a weighted combination of his/her own estimation and the estimation of the immediate neighbors in the last iteration. We use ${\bf{\hat q}}_n^{_{\rm A}}\left( t \right)$ to denote the average energy consumption of all consumers estimated by the consumer $n$ in the $t$th iteration. Then, the aggregated load of the whole network estimated by consumer $n$ is $N{\bf{\hat q}}_n^{_{\rm A}}\left( t \right)$.
  \item \textbf{\emph{Step 2}}: Update his/her energy consumption profile based on the estimated aggregated load through executing a Euclidean projection operation.
  \item \textbf{\emph{Step 3}}: Update his/her own estimation of the average energy consumption.
\end{itemize}

\begin{table}
\begin{center}
\begin{algorithm}[H] \label{}
%\normalsize
\begin{algorithmic}[1]
\small
\STATE Set $t = 1$. Choose any feasible starting point ${\bf{q}}\left(1 \right) = \left( {{{\bf{q}}_n}\left( 1 \right)} \right)_{n = 1}^N$ and set ${\bf{\hat q}}_n^{_{\rm A}}\left( 1 \right)= {{{\bf{q}}_n}\left( 1 \right)}$ for every $n \in {\mathcal N}$. Given the weight parameters $w_{n,k}$ and the step-size $\alpha \left( t \right)$.
\STATE If a suitable termination criterion is satisfied: $\rm{STOP}$.\\
\STATE Each consumer $n \in {\mathcal N}$ updates his/her energy consumption profile and the estimated average energy consumption of all consumers via executing
\begin{align*}
\footnotesize
\begin{split}
&~~~~~~~{\bf{\hat q}}_n^{_{\rm A}}\left( t \right) =  {w_{n,n}}{\bf{\hat q}}_n^{_{\rm A}}\left( t \right) +\sum\nolimits_{k \in {{\cal D}_n}} w _{n,k} {\bf{\hat q}}_k^{_{\rm A}}\left( t \right),\\
&~~~~~~~{{\bf{q}}_n}\left( {t + 1} \right) = {\left[ {{{\bf{q}}_n}\left( t \right) - \alpha \left( t \right){{\bf{F}}_n}\left( {{{\bf{q}}_n}\left( t \right),N{\bf{\hat q}}_n^{_{\rm A}}\left( t \right)} \right)} \right]_{{{\mathcal Q}_n}}},\\
&~~~~~~~{\bf{\hat q}}_n^{_{\rm A}}\left( {t + 1} \right) = {\bf{\hat q}}_n^{_{\rm A}}\left( t \right) + {{\bf{q}}_n}\left( {t + 1} \right) - {{\bf{q}}_n}\left( t \right).
\end{split}
\end{align*}
\STATE $t \leftarrow t + 1$; go to $\rm{STEP}$ 2.
\end{algorithmic}
\caption{: Synchronous Agreement-based Algorithm %(Executed by each consumer $n \in {\mathcal N}$)
}
\end{algorithm}
\end{center}
\end{table}
To proceed, it is convenient to first model the connection topology between consumers. For simplicity but without loss of generality, we model the connection topology of the consumers as an undirected static graph ${\mathcal M}\left({\mathcal N},{\mathcal E}\right)$ with ${\mathcal N}$ being the set of all consumers and ${\mathcal E}$ being the set of undirected edges among the consumers. The notation $\left\{n,k\right\}\in {\mathcal E}$ means that consumer $n$ and consumer $k$ are immediate neighbors, and ${\mathcal D}_n$ denotes the set of consumer $n$'s neighbors, i.e., ${{\cal D}_n} = \left\{ {\left. {k \in {\cal N}} \right|\left\{ {k,n} \right\} \in {\mathcal E}} \right\}$. Now we are ready to present the distributed agreement-based algorithm, which is formally described in {\bf Algorithm 2}, where the notation $w_{n,k}$ denotes the nonnegative weight that consumer $n$ assigns to the estimate of consumer $k$, which is set to zero if $k \notin {{\mathcal D}_n}$ and $n \neq k$.

In terms of the convergence of Algorithm 2, we have the following proposition:
\begin{proposition}\label{prop:converg_of_algorithm2}
Assume that the undirected graph ${\mathcal M}\left({\mathcal N},{\mathcal E}\right)$ is connected, the step-size $\{\alpha(t)\}$ is monotonically decreasing with $t$ and satisfies the following:
\begin{equation}\label{eq:step_size_cond}
\sum\nolimits_{t = 0}^\infty  {\alpha \left( t \right)}  = \infty ,\; {\rm{and}}\;{\sum\nolimits_{t = 0}^\infty  {\left[ {\alpha \left( t \right)} \right]} ^2} < \infty,
\end{equation}
and the weights adopted by the consumers meets the following\footnote{Note that the summations in the following equations are actually equivalent to that over the set ${\mathcal D}_n$. This is because that ${w_{n,k}} = 0$ if consumer $k$ is not a neighbor of consumer $n$.}:
\begin{equation}\label{eq:weight_cond}
\sum\nolimits_{k = 1}^N {{w_{n,k}} = 1}, \forall n ,\; {\rm{and}}\;\sum\nolimits_{n = 1}^N {{w_{n,k}} = 1}, \forall k,
\end{equation}
and the condition in Proposition \ref{Prop:unique_NE} holds. Then the sequence $\left\{{\bf q}(t)\right\}$ generated by Algorithm 2 converges to the unique NE of the formulated game $\mathcal G$.
\end{proposition}
\begin{proof}
See Appendix \ref{appen:proof of prop3}.
\end{proof}
\begin{remark}
In this paper, we use the following formula for the weights \cite[Ch. 4]{Koshal_Thesis_2013}:
\begin{equation}\label{eq:weights}
{w_{n,k}} = \left\{ \begin{array}{l}
 {\tau  \mathord{\left/
 {\vphantom {\tau  {\left[ {{{\max }_n}\left| {{{\cal D}_n}} \right|} \right]}}} \right.
 \kern-\nulldelimiterspace} {\left[ {{{\max }_n}\left| {{{\cal D}_n}} \right|} \right]}}\;\;\;\;\;\;\;\;\;\;\;\;\;\;\;\;{\rm{if}}\;n \ne k \\
 1 - \left| {{{\cal D}_n}} \right|{\tau  \mathord{\left/
 {\vphantom {\tau  {\left[ {{{\max }_n}\left| {{{\cal D}_n}} \right|} \right]}}} \right.
 \kern-\nulldelimiterspace} {\left[ {{{\max }_n}\left| {{{\cal D}_n}} \right|} \right]}}\;\;\;{\rm{if}}\;n = k \\
 \end{array} \right.,
\end{equation}
where $\left| {{{\cal D}_n}} \right|$ denotes the cardinality of the set ${{{\cal D}_n}}$, and $0<\tau<1$ is used to measure the relative proportion of the neighbors' estimates in each consumer's estimation of the average energy consumption. It is straightforward to validate that the weights in (\ref{eq:weights}) satisfy the conditions in (\ref{eq:weight_cond}). Other choices of the weights can be found in \cite{Xiao_2004_J}.

Although information exchanges are required between the consumers in Algorithm 2, the consumers only need to share their estimations of the average energy consumption of all consumers instead of their exact energy consumption profiles with their immediate neighbors. Thus, the developed algorithm can avoid the consumers' security and privacy concerns. $~~\square$
\end{remark}

Note that although the central unit is not necessary for the developed Algorithm 2, synchronization between the consumers and coordination in terms of algorithm step sizes are still required, which are challenging in very large networks. Motivated by this, we will develop a distributed asynchronous algorithm in next section.

\section{Distributed Asynchronous Gossip-based Algorithm without A Central Unit}

In this section, we develop a distributed asynchronous gossip-based algorithm for computing the NE of the formulated game without the need of a central unit. The consumers perform their estimations and updates in the same way as in the Algorithm 2, but the updates occur asynchronously instead of synchronously. The developed algorithm allows the consumers to use uncoordinated step size values. More specifically, the consumers can choose the step size based on their own information-update frequency. The graph model for the connection topology of the consumers in Section \ref{sec:alg_2} is also applicable in this section.

\begin{table}
\begin{center}
\begin{algorithm}[H] \label{}
\begin{algorithmic}[1]
\small
\STATE Set $t = 1$. Choose any feasible starting point ${\bf{q}}\left( 1 \right) = \left( {{{\bf{q}}_n}\left( 1 \right)} \right)_{n = 1}^N$ and set ${\bf{\hat q}}_n^{_{\rm A}}\left( 1 \right)= {{{\bf{q}}_n}\left( 1 \right)}$ for every $n \in {\mathcal N}$. 
\STATE If a suitable termination criterion is satisfied: $\rm{STOP}$.\\
\STATE Each consumer $n \in \left\{I^t,J^t\right\}$ counts the number of updates that he/she has executed up to time $t$ inclusively (denoted by $\epsilon_n(t)$), sets the step size as $\alpha_{n}(t) = 1/\epsilon_n(t)$, and updates his/her energy consumption profile and the estimated average energy consumption of all consumers via executing
\begin{align*}
\footnotesize
\begin{split}
&~~~~~~~{\bf{\hat q}}_n^{_{\rm A}}\left( t \right) =  \frac{{\bf{\hat q}}_{I^t}^{_{\rm A}}\left( t \right) + {\bf{\hat q}}_{J^t}^{_{\rm A}}\left( t \right)}{2},\\
&~~~~~~~{{\bf{q}}_n}\left( {t + 1} \right) = {\left[ {{{\bf{q}}_n}\left( t \right) - \alpha_n \left( t \right){{\bf{F}}_n}\left( {{{\bf{q}}_n}\left( t \right),N{\bf{\hat q}}_n^{_{\rm A}}\left( t \right)} \right)} \right]_{{{\mathcal Q}_n}}},\\
&~~~~~~~{\bf{\hat q}}_n^{_{\rm A}}\left( {t + 1} \right) = {\bf{\hat q}}_n^{_{\rm A}}\left( t \right) + {{\bf{q}}_n}\left( {t + 1} \right) - {{\bf{q}}_n}\left( t \right).
\end{split}
\end{align*}
\STATE $t \leftarrow t + 1$; go to $\rm{STEP}$ 2.
\end{algorithmic}
\caption{: Asynchronous Gossip-based Algorithm %(Executed by each consumer $n \in {\mathcal N}$)
}
\end{algorithm}
\end{center}
\end{table}

To allow for asynchronous updates, we adopt the gossip protocol \cite{Boyd_TIT_2006} to model the consumers' exchange of their estimations for the average energy consumption of all consumers. In this protocol, each consumer is assumed to have a clock which ticks according to a Poisson process with rate 1. At a tick of his/her clock, consumer $n$ contacts a randomly selected\footnote{Here, we consider that each neighbor has an equal chance of being selected.} neighbor $k \in {\mathcal D}_n$ to exchange information. With reference to \cite{Boyd_TIT_2006}, the consumers' clocks processes can be equivalently modeled as a single virtual clock that ticks according to a Poisson process with rate $N$. We assume that only one consumer communicates with its neighbor at each tick of the virtual clock and we use $Z^t$ to denote $t$th tick of the virtual Poisson process. We discretize time so that the instant $t$ corresponds to the time-slot $\left[Z^{t-1}, Z^t\right)$. At each time $t$, every consumer $n$ has his/her consumption profile ${\bf q}_n(t)$ and estimation of the average energy consumption of all consumers ${\bf{\hat q}}_n^{_{\rm A}}\left( t \right)$. Let $I^t \in {\mathcal N} $ denote the consumer whose local clock ticked at time $t$. Note that $I^t$ is uniformly distributed in the set $\mathcal N$ since the Poisson clocks at each consumer are independent. Moreover, the memoryless property of the Poisson arrival process ensures that the process $\left\{I^t\right\}$ is independent and identically distributed. We use $J^t$ to denote the consumer randomly contacted by the consumer $I^t$, where $J^t$ is a neighbor of the consumer $I^t$, i.e., $J^t \in {\mathcal D}_{I^t}$. Then, these two consumers will exchange their estimations of average energy consumption and update their own energy profiles. The developed asynchronous gossip-based algorithm is formally described in {\bf Algorithm 3}. As can be seen from Algorithm 3, the consumers perform the same updates as in the synchronous Algorithm 2, but only two randomly selected consumers update their estimations of average energy consumption and their own energy profiles at each iteration, while the other consumers do not update.

For the convergence of Algorithm 3, we have the following proposition and the proof follows from Appendixes \ref{appen:proof_of_prop1}-\ref{appen:proof of prop3}, the adopted step sizes and \cite[Ch. 4, Prop. 12]{Koshal_Thesis_2013}:
\begin{proposition}\label{prop:convergence_gossip_alg}
Assume that the condition in Proposition \ref{Prop:unique_NE} holds and the undirected graph ${\mathcal M}\left({\mathcal N},{\mathcal E}\right)$ is connected. Then, the sequence of the energy consumption profile $\{{\bf {q}}(t)\}$ generated by Algorithm 3 converges to the unique NE of the game ${\mathcal G}$ almost surely.
\end{proposition}
\begin{remark}
In this developed algorithm, synchronization is not required between the consumers. Besides the asynchronous updates, the consumers are allowed to use uncoordinated step sizes that are based on the frequency of the consumers' updates. Specifically, consumer $n$ uses the step size
\begin{equation}\label{}
 \alpha_n(t) = \frac{1}{\epsilon_n(t)},~n \in \left\{I^t,J^t\right\},
\end{equation}
at the $t$th iteration, where ${\epsilon_n(t)}$ denotes the numbers of updates that consumer $n$ has performed up to time $t$ inclusively. In addition, analogous to Algorithm 2, no private information is required to exchange between the consumers in Algorithm~3.

It is worth mentioning that the pairwise gossip protocol (i.e, only a random pair of
consumers is chosen to update at each iteration) is adopted for simplicity in this paper. The developed algorithm can be extended to the general setup where a random subset of consumers (more than one pair) exchange their estimations and update their energy profiles at each iteration. This will be considered in our future work. $\square$
\end{remark}

\section{Numerical Results}

In this section, we present some numerical results to validate the above theoretical analysis and illustrate the performance of the developed algorithms.

In the following simulation results, we consider the residential scenario consisting of $N=50$ consumers, where the consumers determine their energy consumption for the following whole day, which starts from 8 AM. Each time slot is set as one hour, i.e., $H = 24$ and the first time slot corresponds to the hour between 8 AM and 9 AM. In Fig.\ \ref{fig1}, we provide a typical energy consumption interval of a residential consumer \cite[Figs. 2.5-2.7]{Willis_book_2004},\cite{Gatsis_TSG_2012}. Considering that different consumers may have different energy consumption interval in practice, the `Low limit' and `Upper limit' of each consumer in each time slot are formed by respectively adding a random real number to the corresponding value in Fig.\ \ref{fig1}. Then, the initial energy consumption of a certain consumer in each time slot, $q_n^{h}(1)$, is uniformly chosen between his/her corresponding `Low limit' and `Upper limit'. The numerical results show that the selected consumption parameters yield the total energy consumption of every consumer in the order of 10 kWh to 30 kWh, which is representative of a residential consumer \cite{Gatsis_TSG_2012}.
\begin{figure}
\centering \scalebox{0.5}{\includegraphics{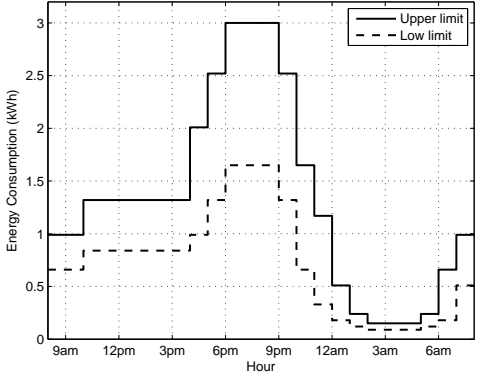}}
\caption{The typical energy consumption interval of a residential consumer. \label{fig1}}
\end{figure}
\begin{figure}
\centering \scalebox{0.5}{\includegraphics{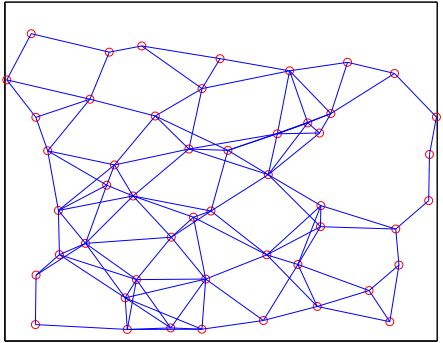}}
\caption{The connection topology for the consumers used in Algorithm 2 and 3. The consumers are denoted by the circles and their connections are represented by the solid lines. \label{fig2}}
\end{figure}

According to Fig. \ref{fig1}, we classify the whole time horizon into three segments: off-peak hours (12 AM to 7 AM), mid-peak hours (7 AM to 4 PM and 10 PM to 12 AM), and on-peak hours (4 PM to 10 PM). We also set $a_h$ equal to 0.003, 0.004 and 0.005 for the off-peak, mid-peak and on-peak hours, respectively, and parameters $b_h$ and $c_h$ are set equal to 1.2 and 0 for $\forall h \in \mathcal{H}$, respectively. In the considered DSM scenario, the value $q_n^{h,\min}$ for the $n$th consumer is set to his/her `Low limit' of the $h$th time slot. In addition, the values of $q_n^{h,\max}$ for mid-peak and on-peak hours are set to his/her maximum value of the `Upper limit', while the values of $q_n^{h,\max}$ for the off-peak hours are uniformly chosen from the interval $[0.4,0.6]$. The values of $E_n$ are chosen to be equal to the sum of the consumers' initial energy consumption profiles before applying the DSM program. Moreover, the parameters for the algorithms are chosen\footnote{We refer the readers to \cite{Kannan_SIAM_2012} for more discussion on the choice of algorithm parameters.} as follows: $\gamma \left( t \right) = {t^{ - 0.51}}$ and $\theta  = 0.2$ for Algorithm 1, and $\alpha \left( t \right) = {t^{ - 0.51}}$ and $\tau = 0.5$ for Algorithm 2. Finally, a randomly generated connection structure of the consumers for Algorithm 2 and 3 are given in Fig. \ref{fig2}, where two consumers are directly linked means that they are immediate neighbors, who can exchange information in the iterations of the algorithms.

\begin{figure*}
\centering
 \subfigure[Convergence of Algorithms 1 and 2]
  {\scalebox{0.5}{\includegraphics {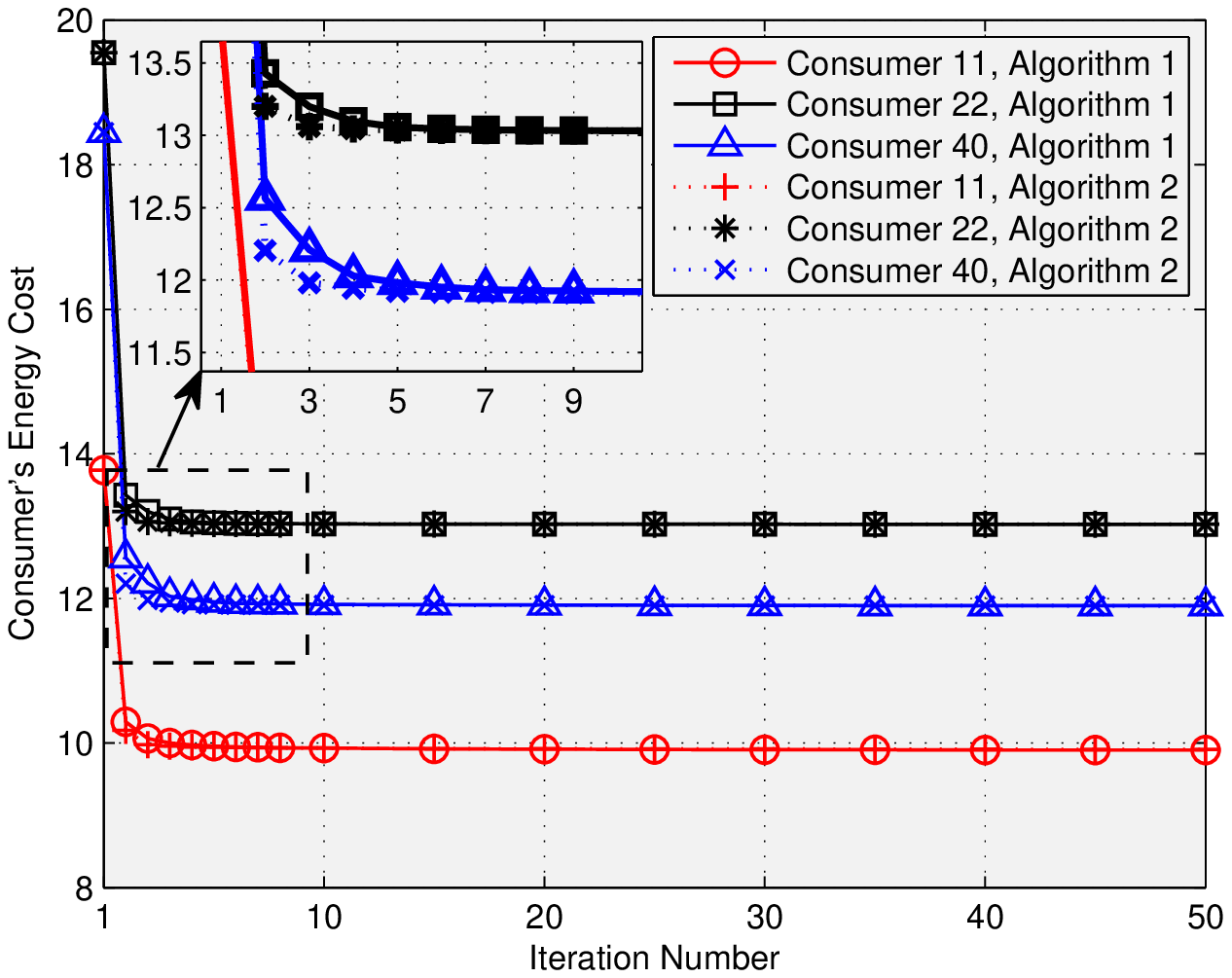}
  \label{fig3_a}}}
\hfil
 \subfigure[Convergence of Algorithm 3]
  {\scalebox{0.5}{\includegraphics {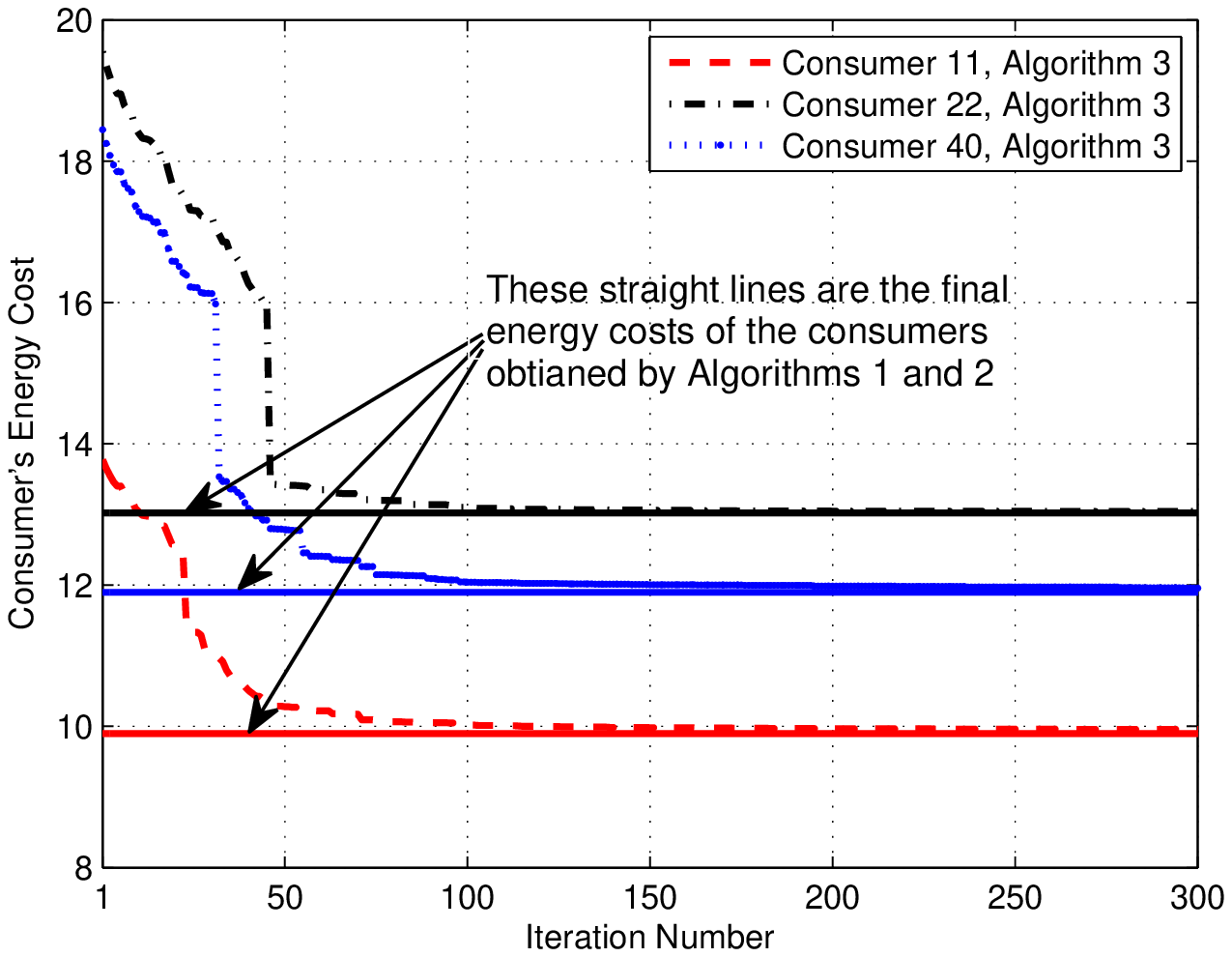}
\label{fig3_b}}}
\caption{The convergence of the developed algorithms in terms of the consumers' total energy cost.}
\label{fig3}
\end{figure*}
\begin{figure*}
\centering \scalebox{0.7}{\includegraphics{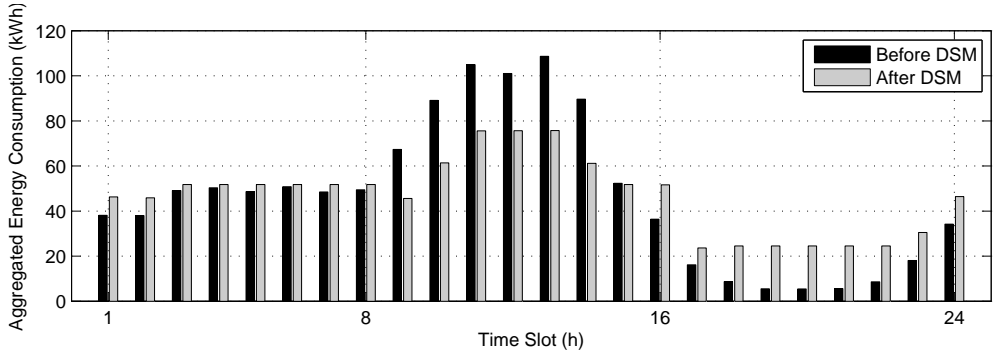}}
\caption{The aggregated energy consumption profiles of all consumers before and after the DSM program. \label{fig4}}
\end{figure*}

Fig. \ref{fig3} plots the total energy cost for three different consumers versus the number of iterations of the developed algorithms. It can be observed from Fig. \ref{fig3} (a) that both Algorithm 1 and Algorithm 2 converge to the NE of the formulated game very quickly. Specifically, the energy cost of each consumer has already  achieved a relatively stable state after the first 10 iterations, which verifies the validness of both Proposition \ref{prop:convergence_alg_1} and Proposition \ref{prop:converg_of_algorithm2}, as well as displaying the high efficiency of the developed algorithms. Fig. \ref{fig3} (b) is plotted to illustrate the convergence performance of the Algorithm 3. As can be observed from Fig. \ref{fig3} (b) that the total energy cost of different consumers approach to coincide with that obtained by Algorithm 1 and 2 after $200$ iterations. This validates the results given in Proposition \ref{prop:convergence_gossip_alg}. Note that due to space limitations, we only show results in Fig. \ref{fig3} for three randomly selected consumers, although it can be shown that similar results also hold for the other consumers and a wide range of settings with different parameters.

In Fig \ref{fig4}, we compare the aggregated energy consumption profiles of all consumers corresponding to the situations before and after DSM program. We clearly observe from Fig. \ref{fig4} that the proposed DSM scheme effectively encourages the consumers to shift their energy consumption from peak to non-peak hours. We also investigate the peak-to-average ratio (PAR) of the aggregated load defined as
\begin{equation}
{\rm{PAR}} = \frac{{H{{\max }_h}q_\Sigma ^h}}{{\sum\nolimits_{h = 1}^H {q_\Sigma ^h} }}.
\end{equation}
The simulation results show that the PAR decreases from $2.3189$ to $ 1.6161$ (i.e., $30.31$\% less) before and after the DSM program. This will result in a generally flattened demand profile, which will not only reduce the consumers' energy cost but also benefit the efficiency of the whole power grid.

To show that the adopted billing method can fairly charge the consumers, we plot the energy consumption profiles of consumer $43$ and consumer $50$ after applying the proposed DSM program in Fig. \ref{fig5}. Their total daily energy requirements are $E_{43} = 20.63$~(kWh) and $E_{50} = 19.99$~(kWh), respectively. If the total load billing method \cite{Mohsenian-Rad_TSG_2010} was used, consumer $43$ would be charged more than consumer $50$ since $E_{43} > E_{50}$. However, as can be observed from Fig. \ref{fig5}, the on-peak energy usage of consumer $50$ is larger than that of consumer $43$. This can also be reflected by the PAR values of these two consumers, i.e., ${\rm{PAR}}_{50} = 1.9438$ and ${\rm{PAR}}_{43} = 1.7863$. Thus, it may be not fair to charge consumer $43$ more than consumer $50$ simply because he/she consumes more energy totally. In contrast, our numerical results show that consumer $43$ and consumer $50$ will finally be charged ${\mathcal B}_{43} = 10.56$ and ${\mathcal B}_{50} = 10.66$ (i.e., ${\mathcal B}_{43}<{\mathcal B}_{50}$) after the proposed DSM program. This result is understandable since the adopted billing method considers not only how much the consumers consume the energy totally but also when the consumers use the energy. By this example, we show that the adopted billing approach can charge the consumers more fairly, thereby motivating the consumers to participate in the DSM program.

Fig. \ref{fig6} compares the total energy cost of all consumers for three different cases with different number of consumers. As expected, it can be observed from Fig. \ref{fig6} that the total energy cost is significantly reduced after the proposed DSM program. We also compare the performance of the proposed game-theoretical DSM program with the optimal one obtained by solving the following social welfare optimization problem:
\begin{equation}\label{eq:centralized_opt}
\begin{array}{l}
 \mathop {\min }\limits_{\left\{ {{{\bf{q}}_1}, \ldots ,{{\bf{q}}_N}} \right\}} \;\sum\nolimits_{n = 1}^N {{{\mathcal B}_n}\left( {{{\bf{q}}_n},{{\bf{q}}_{ - n}}} \right)}  \\
 \;\;\;\;\;\;\;{\rm{s}}{\rm{.t}}{\rm{.}}\;\;{{\bf{q}}_n} \in {{\mathcal Q}_n},\;\forall n \in {\mathcal N} \\
 \end{array}.
\end{equation}
From Fig. \ref{fig6}, we can observe that the total energy cost achieved by the proposed DSM program is almost the same with the optimal one\footnote{The theoretical analysis of this observation (i.e., the price of anarchy analysis for the formulated game) is out of the scope of this paper and will be considered in future work.}. Thus, we can claim that the proposed DSM framework qualifies as a practically appealing candidate for the DSM of future smart grid.
\begin{figure}[!ht]
\centering \scalebox{0.48}{\includegraphics{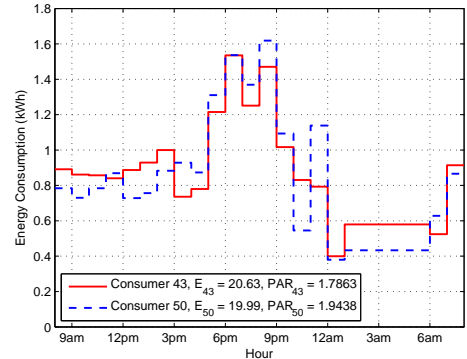}}
\caption{The energy consumption profiles of consumer 43 and consumer 50 after the DSM program. \label{fig5}}
\end{figure}
\begin{figure}[!ht]
\centering \scalebox{0.48}{\includegraphics{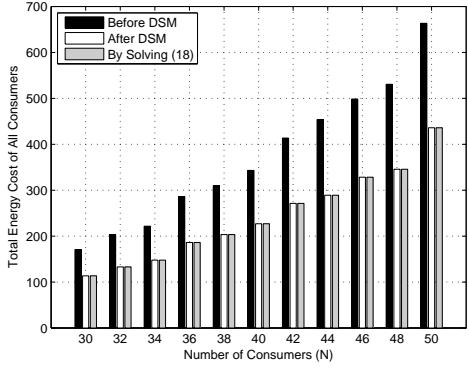}}
\caption{Comparison of total energy cost of all consumers before the DSM program, after the DSM program, and obtained by social welfare optimization. \label{fig6}}
\end{figure}

\section{Conclusions}
In this paper, we formulated an aggregative game for the demand side management program based on energy consumption scheduling and instantaneous load billing, where the consumers are selfish and compete to minimize their individual energy cost. The sufficient condition for the existence and uniqueness of Nash equilibrium (NE) of the formulated game was subsequently given and proved. Based on the formulation of the aggregative game, we developed three distributed algorithms to achieve the NE of the formulated game, corresponding to the two scenarios where the consumers can or cannot access the real-time information of the aggregated load. In these algorithms, the choice for the algorithm parameters do not depend on the arguments of the system and no private information is required to exchange between consumers. Numerical results showed that the developed algorithms can quickly converge to the NE of the formulated game and efficiently convince the consumers to shift their on-peak consumption, which are beneficial to both the consumers and the whole~grid.

\begin{appendix}
\subsection{Proof of Lemma \ref{lemma:property_set_function}}\label{appen:proof of lemma1}
It is evident that the statements in the first part of Lemma \ref{lemma:property_set_function} holds. Hence, we only need to prove the convexity of ${{\cal B}_n}\left( {\cdot\;,\;\cdot\; + \sum\nolimits_{m = 1,m \ne n}^N {{{\bf{q}}_m}} } \right)$ in ${{\bf{q}}_n}$ for every fixed ${{\bf{q}}_{-n}}$. This can be achieved by proving that the Hessian of ${{\cal B}_n}\left( {{{\bf{q}}_n},{{\bf{q}}_\Sigma }} \right)$ is positive semidefinite \cite{Boyd_Convex_2004}. After some algebraic manipulation, we have
\begin{equation}\label{Hessian}
\begin{split}
\nabla _{{{\bf{q}}_n}}^2{{\cal B}_n}\left( {{{\bf{q}}_n},{{\bf{q}}_\Sigma }} \right) = {\rm{diag}}\left\{ {\left[ {q_n^h{p_h}''\left( {{q_\Sigma ^h}} \right) + 2{p_h}'\left( {{q_\Sigma ^h}} \right)} \right]_{h = 1}^H} \right\}.
\end{split}
\end{equation}
Since (\ref{Hessian}) is a diagonal matrix with all diagonal elements being positive, the Hessian matrix of ${{\mathcal {B}}_n}\left( {{{\bf{q}}_n},{{\bf{q}}_{ - n}}} \right)$ is positive semidefinite. This completes the proof. 
\subsection{Proof of Proposition \ref{Prop:unique_NE}}\label{appen:proof_of_prop1}
Based on Lemma \ref{lemma:eq_game_VI}, the proof of this proposition follows if we can show that the formulated VI$\left({\mathcal Q},{\bf F}\right)$ in Lemma \ref{lemma:eq_game_VI} only possesses one solution.  With reference to \cite[Thm. 4.1]{Scutari_Springer_2012}\cite[Thm. 2.3]{Cominetti_book_2012}, the VI$\left({\mathcal Q},{\bf F}\right)$ admits a unique solution if the mapping ${\bf{F}}\left( {\bf{q}} \right)$ is strictly monotone\footnote{A mapping $\bf F\left( {\bf{x}} \right)$: ${\mathcal K}\ni {\bf{x}} \rightarrow\mathbb{R}^N$ is said to be \emph{strictly monotone} on ${\mathcal K}$ if $ \left( {{\bf{x}} - {\bf{y}}} \right)^T{\left( {{\bf{F}}\left( {\bf{x}} \right) - {\bf{F}}\left( {\bf{y}} \right)} \right)} > 0, ~\forall \bf x, \bf y \in {\mathcal K}\; \rm{and} \; \bf x \neq \bf y$.} over $\mathcal Q$ since the feasible set $\mathcal Q$ is compact and convex.

To prove the strict monotonicity of the mapping ${\bf{F}}\left( {\bf{q}} \right)$, it suffices to show that 
\begin{equation}\label{def_eq_1}
\begin{split}
&\sum\nolimits_{h = 1}^H {\sum\nolimits_{n = 1}^N {\left[ {\left( {q_n^h - s_n^h} \right)\left( {{\nabla _{q_n^h}}{{\mathcal B}_n}\left( {\bf{q}} \right) - {\nabla _{s_n^h}}{{\mathcal B}_n}\left( {\bf{s}} \right)} \right)} \right]} } >0 ,\\
\end{split}
\end{equation}
for any ${\bf{q}} = \left( {{{\bf{q}}_n}} \right)_{n = 1}^N,\; {\bf{s}} = \left( {{{\bf{s}}_n}} \right)_{n = 1}^N \in {\cal Q}$.

Let ${{\bf{l}}^h} = \left( {q_1^h, \ldots ,q_N^h} \right)^T$ and ${{\bf{j}}^h} = \left( {s_1^h, \ldots ,s_N^h} \right)^T$, then we can be re-write (\ref{def_eq_1}) as
\begin{equation}\label{def_eq_2}
\begin{split}
&\sum\nolimits_{h = 1}^H {\left[ {\left( {{{\bf{l}}^h} - {{\bf{j}}^h}} \right)^T\left( {{\nabla _{{{\bf{l}}^h}}}{\mathcal B}_n^h\left( {{{\bf{l}}^h}} \right) - {\nabla _{{{\bf{j}}^h}}}{\mathcal B}_n^h\left( {{{\bf{j}}^h}} \right)} \right)} \right]}>0,
\end{split}
\end{equation}
where ${\mathcal B}_n^h\left( {{{\bf{l}}^h}} \right) = {p_h}\left( {q_{\Sigma}^h} \right)q_n^h$, and ${\nabla _{{{\bf{l}}^h}}}{\cal B}_n^h\left( {{{\bf{l}}^h}} \right) = {\left( {{\nabla _{q_1^h}}{\cal B}_1^h\left( {{{\bf{l}}^h}} \right),{\nabla _{q_2^h}}{\cal B}_2^h\left( {{{\bf{l}}^h}} \right), \ldots ,{\nabla _{q_N^h}}{\cal B}_n^h\left( {{{\bf{l}}^h}} \right)} \right)^T}$.

We observe that a sufficient condition for (\ref{def_eq_2}) to hold is if
\begin{equation}\label{def_eq_6}
\left( {{{\bf{l}}^h} - {{\bf{j}}^h}} \right)^T\left[ {{\bf g}_h\left( {{{\bf{l}}^h}} \right) -{\bf g}_h\left( {{{\bf{j}}^h}} \right)} \right] > 0,\;\forall h \in {\mathcal H},
\end{equation}
where ${{\bf{g}}_h}\left( {{{\bf{l}}^h}} \right) = {\nabla _{{{\bf{l}}^h}}}{\cal B}_n^h\left( {{{\bf{l}}^h}} \right) $, which is defined for the sake of notation.

Recall the definition of a strictly monotone mapping, we can obtain that (\ref{def_eq_6}) holds if the mapping ${{\bf{g}}_h}\left( {{{\bf{l}}^h}} \right)$ is strictly monotone. With reference to \cite[Eq. (4.8)]{Scutari_Springer_2012}, the condition in (\ref{def_eq_6}) can be shown to be equivalent to proving the Jacobian matrix of ${{\bf{g}}_h}\left( {{{\bf{l}}^h}} \right)$ is positive definite. Since the transpose operation does not change the definite property of a given matrix, what we only to prove is that the transpose of the Jacobian matrix of ${{\bf{g}}_h}\left( {{{\bf{l}}^h}} \right)$, denoted by ${{\bf{G}}_h}\left( {{{\bf{l}}^h}} \right) = {\nabla _{{{\bf{l}}^h}}}{{\bf{g}}_h}\left( {{{\bf{l}}^h}} \right)$, is positive definite. 

To proceed, we have the $(n,m)$th entry of ${{\bf{G}}_h}\left( {{{\bf{l}}^h}} \right)$ after some algebraic manipulation given by
\begin{equation}\label{def_eq_7}
{\left[ {{{\bf{G}}_h}\left( {{{\bf{l}}^h}} \right)} \right]_{n,m}} = \left\{ \begin{array}{l}
\sigma_h\left[ {2{q_{\Sigma}^h} + \left( {{b_h} - 1} \right)q_n^h} \right]
,\;{\mathop{\rm if}\nolimits} \;n = m\\
\sigma_h\left[ {{q_{\Sigma}^h} + \left( {{b_h} - 1} \right)q_n^h} \right] ,\;\;\;{\mathop{\rm if}\nolimits} \;n \ne m
\end{array}\right.
\end{equation}
where $\sigma_h =  {a_h}{b_h}{\left( {{q_{\Sigma}^h}} \right)^{{b_h} - 2}}$.

Since the matrix ${{\bf{G}}_h}\left( {{{\bf{l}}^h}} \right)$ may not be symmetric, we can prove its positive definiteness by showing that the symmetric matrix
\begin{equation}\label{def_eq_8}
{{\bf{G}}_h}\left( {{{\bf{l}}^h}} \right) + {{\bf{G}}_h}{\left( {{{\bf{l}}^h}} \right)^T} = \sigma_h\left( {\underbrace {{{\bf{z}}^h}{\bf{1}}^T + {{\bf{1}}}\left({{\bf{z}}^h}\right)^T}_{{{\mathcal T}_b}} + 2{q_{\Sigma}^h}{\bf{I}}} \right),
\end{equation}
is positive definite \cite{Altman_TAC_2002}, where ${{\bf{z}}^h} = {q_{\Sigma}^h}{\bf{1}} + \left( {b_h  - 1} \right){{\bf{l}}^h}$, ${\bf{1}}$ is a $N \times 1$ vector with every element of 1. This is equivalent to showing that the smallest eigenvalue of this matrix is positive.

After some appropriate calculations \cite{Altman_TAC_2002}, the two non-zero eigenvalues of the matrix ${{\mathcal T}_b}$ in (\ref{def_eq_8}) are given by
\begin{equation}\label{def_eq_9}
\begin{split}
&\eta _{{{\mathcal T}_b}}^1 = \left( {N + 1{\rm{ + }}{b_h}} \right){q_{\Sigma}^h} + \sqrt {N\left({{\bf{z}}^h}\right)^T{{\bf{z}}^h}}  - 2{q_{\Sigma}^h},\\
&\eta _{{{\mathcal T}_b}}^2 = \left( {N + 1{\rm{ + }}{b_h}} \right){q_{\Sigma}^h} - \sqrt {N\left({{\bf{z}}^h}\right)^T{{\bf{z}}^h}} - 2{q_{\Sigma}^h}.
\end{split}
\end{equation}
Note that $N\ge2$ is implicit here.

Since $\eta _{{{\mathcal T}_b}}^1 \ge \eta _{{{\mathcal T}_b}}^2$, the smallest eigenvalue of the matrix ${{\bf{G}}_h}\left( {{{\bf{l}}^h}} \right) + {{\bf{G}}_h}{\left( {{{\bf{l}}^h}} \right)^T}$ can be expressed as
\begin{equation}\label{def_eq_10}
{\eta _{\min }} = \sigma_h \min \left( {\left( {N + 1{\rm{ + }}{b_h}} \right){q_{\Sigma}^h} - \sqrt {N\left({{\bf{z}}^h}\right)^T{{\bf{z}}^h}} ,2{q_{\Sigma}^h}} \right) ,%- 2c_{\rm{sm}},
\end{equation}
where the second $2{q_{\Sigma}^h}$ term in the $\min$ function arises because there are $N-2$ zero eigenvalues in the matrix ${{\mathcal T}_b}$.

To further simplify (\ref{def_eq_10}), we have
\begin{equation}\label{def_eq_11}
\begin{split}
\left({{\bf{z}}^h}\right)^T{{\bf{z}}^h} &= {\sum\nolimits_{n = 1}^N {\left( {{q_{\Sigma}^h} + \left( {{b_h} - 1} \right)q_n^h} \right)} ^2}\\
&\le \left( {N - 1 + {{\left( {{b_h}} \right)}^2}} \right){\left( {{q_{\Sigma}^h}} \right)^2},
\end{split}
\end{equation}

Substituting (\ref{def_eq_11}) into (\ref{def_eq_10}), we obtain
\begin{equation}\label{def_eq_12}
{\eta _{\min }} \ge \sigma_h\min \left( {\kappa_h {q_{\Sigma}^h},2{q_{\Sigma}^h}} \right),
\end{equation}
where $\kappa_h  = \left( {N + 1{\rm{ + }}{b_h}} \right) - \sqrt {N\left( {N - 1 + {{\left( {{b_h}} \right)}^2}} \right)} $.

Since $\sigma_h > 0$, we observe from the right hand side of (\ref{def_eq_12}) that ${\eta _{\min }} > 0$ if $\kappa_h >0$, or equivalently ${\left( {N + 1{\rm{ + }}{b_h}} \right)^2} > N\left( {N - 1 + {{\left( {{b_h}} \right)}^2}} \right)$, which we can re-written as $\left( {1{\rm{ + }}{b_h}} \right)\left( {\left( {N - 1} \right){b_h} - \left( {3N + 1} \right)} \right) < 0$. Thus, a sufficient condition for $\kappa_h > 0$ is
\begin{equation}\label{def_eq_16}
{b_h} < 3 + 4/\left( {N - 1} \right).
\end{equation}
This completes the proof.

\subsection{Proof of Proposition \ref{prop:convergence_alg_1}}\label{appen:proof of prop2}
The proof for the convergence of a general iterative proximal-point algorithm and the corresponding necessary conditions were presented in \cite[Sec. 3]{Kannan_SIAM_2012}. Therefore, we only need to prove that the formulated NEP ${\mathcal G}_{\pmb \lambda}$ meets all the required conditions listed in \cite[Assumption (A3)]{Kannan_SIAM_2012}.

Firstly, it is evident that the set $\mathcal Q$ is compact, and that $\left\| {\bf{q}} \right\|_2$ and $\left\| {\bf F}\left(\bf q\right) \right\|_2$ are both bounded for $\forall {\bf{q}} \in \mathcal Q$. Secondly, as we have proved in Appendix \ref{appen:proof_of_prop1}, the mapping ${\bf F}\left(\bf q\right)$ is strictly monotone on $\bf q$ when the price parameter ${b_h}$ satisfies ${b_h} < 3 + 4/\left( {N - 1} \right)$ for $ \forall h \in {\mathcal H}$. Therefore, we only need to prove that ${\bf F}\left(\bf q\right)$ is Lipschitz continuous over $ {\mathcal Q}$, i.e., show that there exists a real constant ${c_1^{\rm{lip}}} >0$ such that, for all $\bf {q,~s} \in {\mathcal Q}$,
\begin{equation}\label{eq:lc_condition_1}
{\left\| {{\bf{F}}\left( {\bf{q}} \right) - {\bf{F}}\left( {\bf{s}} \right)} \right\|_2} \le {c_1^{\rm{lip}}}{\left\| {{\bf{q}} - {\bf{s}}} \right\|_2} \; .
\end{equation}

According to the definition of Euclidean norm, (\ref{eq:lc_condition_1}) can be shown to hold  if each element of the function ${\bf{F}}\left( {\bf{q}} \right)$, denoted by $f_n^h\left( {\bf{q}} \right) = d{\mathcal B}_n\left( {\bf{q}} \right)/dq_n^h$, is Lipschitz continuous in $\bf q$, i.e., for any $n \in{\mathcal N}$ and any $h \in{\mathcal H}$, there exists a real constant ${c_{n,h}^{\rm{lip}}} >0$ such that, for all $\bf {q,~s} \in {\mathcal Q}$,
\begin{equation}\label{eq:lc_condition_2}
\left| {f_n^h\left( {\bf{q}} \right) - f_n^h\left( {\bf{s}} \right)} \right| \le c_{n,h}^{{\rm{lip}}}{\left\| {{\bf{q}} - {\bf{s}}} \right\|_2}.
\end{equation}

We now proceed to prove (\ref{eq:lc_condition_2}) holds. After some algebraic manipulation, we have
\begin{equation}\label{eq:F_function_element}
f_n^h\left( {\bf{q}} \right) = {p_h}'\left( {q_\Sigma ^h} \right)q_n^h + {p_h}\left( {q_\Sigma ^h} \right) = f_n^h\left( {q_n^h,q_\Sigma ^h} \right).
\end{equation}
Substituting (\ref{eq:F_function_element}) into the left-hand side of (\ref{eq:lc_condition_2}), we have
\begin{equation}\label{eq:lc_condition_3}
\begin{split}
\left| {f_n^h\left( {\bf{q}} \right) - f_n^h\left( {\bf{s}} \right)} \right| =& \left| {f_n^h\left( {q_n^h,q_\Sigma ^h} \right) - f_n^h\left( {s_n^h,s_\Sigma ^h} \right)} \right| \\
\le &\left| {f_n^h\left( {q_n^h,q_\Sigma ^h} \right) - f_n^h\left( {s_n^h,q_\Sigma ^h} \right)} \right|+  \\
&\left| {f_n^h\left( {s_n^h,q_\Sigma ^h} \right) - f_n^h\left( {s_n^h,s_\Sigma ^h} \right)} \right|,
\end{split}
\end{equation}
where $s_\Sigma ^h = \sum\nolimits_{n = 1}^N {s_n^h}$ and the inequality follows according to the triangular inequality.

Now by recalling that ${p_h}\left( {{q_\Sigma ^h}} \right) = a_h {\left( q_\Sigma ^h \right)^{b_h}} + c_h$, we can rewrite the function ${f_n^h\left( {q_n^h,q_\Sigma ^h} \right)} $ as
\begin{equation}\label{eq:lipcond2}
\begin{split}
f_n^h\left( {q_n^h,q_\Sigma ^h} \right) = {a_h}{\left( {q_\Sigma ^h} \right)^{{b_h} - 1}}\left( {{b_h}q_n^h + q_\Sigma ^h} \right) + {c_h}.
\end{split}
\end{equation}

With reference to \cite[Ch. 12]{Eriksson_Mathbook_2004}, it is straightforward to deduce that the function ${f_n^h\left( {q_n^h,q_\Sigma ^h} \right)} $ in (\ref{eq:lipcond2}) is Lipschitz continuous in $q_n^h$ for a fixed $q_\Sigma ^h$ and is Lipschitz continuous in $q_\Sigma ^h$ for a fixed $q_n^h$. That is there exists two real constant $c_{n,h,1}^{\rm{lip}}, c_{n,h,2}^{\rm{lip}}>0$ such that for any $q_n^h$ and $s_n^h$,
\begin{equation}\label{eq:lc_condition_4}
\begin{split}
&\left| {f_n^h\left( {q_n^h,q_\Sigma ^h} \right) - f_n^h\left( {s_n^h,q_\Sigma ^h} \right)} \right| \le c_{n,h,1}^{\rm{lip}}\left| {q_n^h - s_n^h} \right|\\
&~~~~~~~~~~~~~~~~= c_{n,h,1}^{\rm{lip}} \sqrt {{{\left( {q_n^h - s_n^h} \right)}^2}}  \le c_{n,h,1}^{\rm{lip}} {\left\| {{\bf{q}} - {\bf{s}}} \right\|_2},
\end{split}
\end{equation}
and for any $q_\Sigma ^h$ and $s_\Sigma ^h$,
\begin{equation}\label{eq:lc_condition_5}
\begin{split}
&\left| {f_n^h\left( {s_n^h,q_\Sigma ^h} \right) - f_n^h\left( {s_n^h,s_\Sigma ^h} \right)} \right| \le c_{n,h,2}^{\rm{lip}}\left| {q_\Sigma ^h - s_\Sigma ^h} \right|\\
&\le c_{n,h,2}^{\rm{lip}}\sqrt {\sum\nolimits_{n = 1}^N {{{\left( {q_n^h - s_n^h} \right)}^2}} }
\le c_{n,h,2}^{\rm{lip}} {\left\| {{\bf{q}} - {\bf{s}}} \right\|_2} .
\end{split}
\end{equation}

By substituting (\ref{eq:lc_condition_4}) and (\ref{eq:lc_condition_5}) into (\ref{eq:lc_condition_3}), we deduce that it can always find a real constant ${c_{n,h}^{\rm{lip}}} >0$ such that (\ref{eq:lc_condition_2}) holds for all $\bf {q,~s} \in {\mathcal Q}$. This completes the proof.

\subsection{Proof of Proposition \ref{prop:converg_of_algorithm2}}\label{appen:proof of prop3}
The sufficient conditions (i.e., \cite[Ch. 4, Assumptions 8-13]{Koshal_Thesis_2013}) and the rigorous proofs for the convergence of a general distributed agreement-based algorithm have been provided in \cite[Ch. 4.1]{Koshal_Thesis_2013}. Hence, the proof of the Proposition \ref{prop:converg_of_algorithm2} follows if the formulated aggregative game meets all the required conditions listed in \cite[Ch. 4, Assumptions 8-13]{Koshal_Thesis_2013}. Based on the analysis in Lemma \ref{lemma:property_set_function} and Proposition \ref{Prop:unique_NE}, the adopted structure of the weights and the assumption on the step-size, we can claim that the considered aggregative game has already satisfied all the conditions except the one stated in \cite[Ch. 4, Assumptions 10]{Koshal_Thesis_2013}. Therefore, the remaining task is to prove that the formulated game also meets the condition in \cite[Ch. 4, Assumptions 10]{Koshal_Thesis_2013}. More specifically, we need to show that each mapping ${{\bf{F}}_n}\left( {{{\bf{q}}_n},{{\bf{q}}_\Sigma }} \right)$ is Lipschitz continuous in ${{\bf{q}}_\Sigma }$ for every fixed ${\bf q}_n \in {\mathcal Q}_n$.

Analogous to the analysis in Appendix \ref{appen:proof of prop2}, ${{\bf{F}}_n}\left( {{{\bf{q}}_n},{{\bf{q}}_\Sigma }} \right)$ is Lipschitz continuous in ${{\bf{q}}_\Sigma }$ if each element of this function, $f_n^h\left( {{{\bf{q}}_n},{{\bf{q}}_\Sigma }} \right)$, is Lipschitz continuous in ${{\bf{q}}_\Sigma }$. The validity for the Lipschitz continuity of $f_n^h\left( {{{\bf{q}}_n},{{\bf{q}}_\Sigma }} \right)$ follows since $f_n^h\left( {{{\bf{q}}_n},{{\bf{q}}_\Sigma }} \right) = f_n^h\left( {q_n^h,q_\Sigma ^h} \right)$ and ${f_n^h\left( {q_n^h,q_\Sigma ^h} \right)} $ is Lipschitz continuous in $q_\Sigma ^h$ for a fixed $q_n^h$ (cf. Appendix \ref{appen:proof of prop2}). This completes the proof.

\end{appendix}
\section{Acknowledgement}
The authors would like to thank the anonymous reviewers for
their valuable comments and suggestions, which improved the
quality of the paper. The authors also thank Dr. Gregor Verbic,
Prof. David Hill, Dr. Archie Chapman and Dr. Peng Wang for their helpful discussion.

\ifCLASSOPTIONcaptionsoff
  \newpage
\fi

\bibliographystyle{IEEEtran}
\bibliography{endnote}

\begin{IEEEbiography}[{\includegraphics[width=1in,height=1.25in,clip,keepaspectratio]
{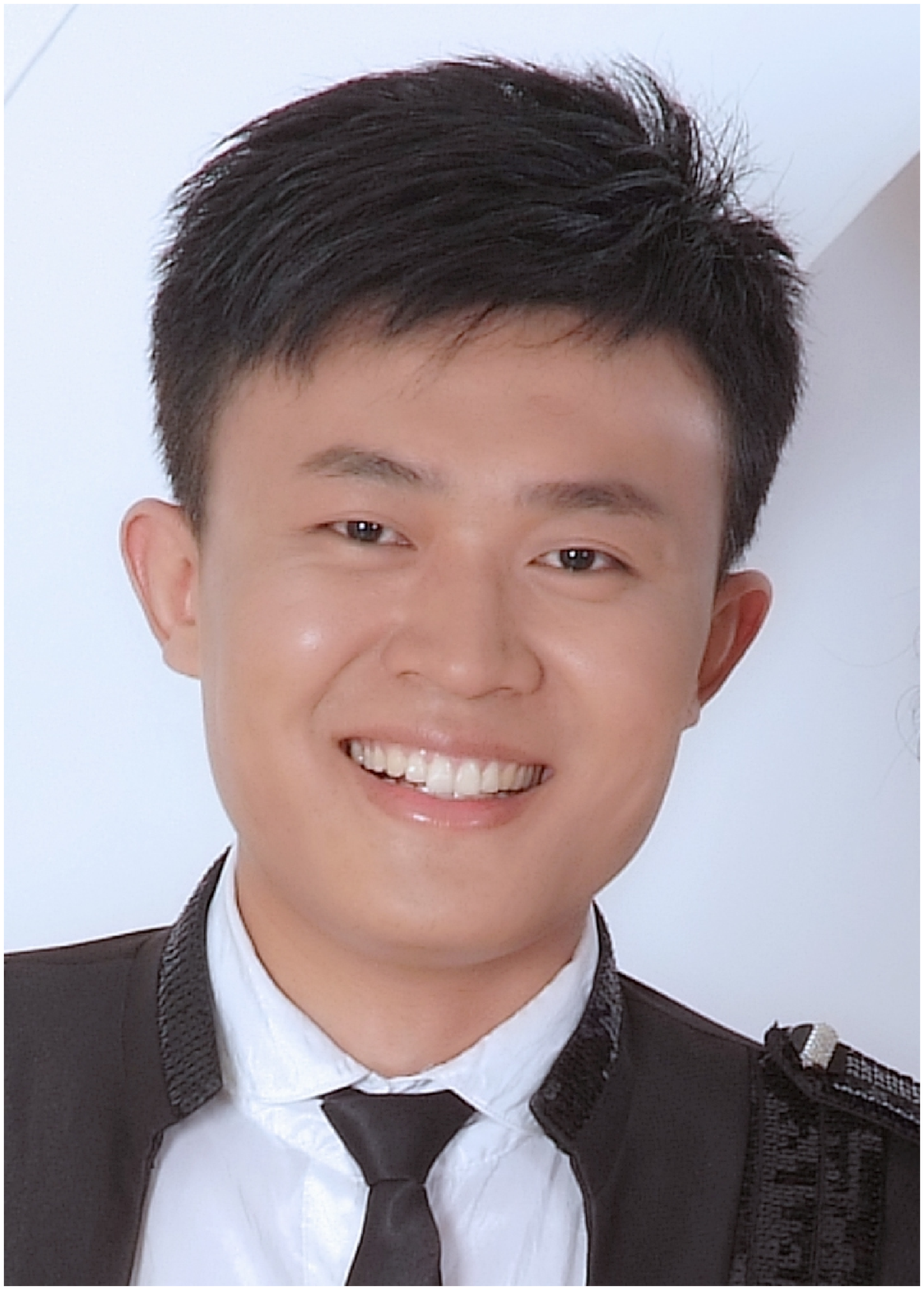}}]{He (Henry) Chen}
(S'10) received his B.E. degree in Communication Engineering and M.E. degree (research) in Communication and Information System both from Shandong University, China, in 2008 and 2011, respectively. He was awarded the Outstanding Bachelor Thesis of Shandong University in 2008 and the Outstanding Master Thesis of Shandong Province in 2012. He is currently working towards the Ph.D. degree in electrical engineering at the University of Sydney, Sydney, Australia.

His current research interests include demand side management of smart grid, wireless communications powered by wireless energy transfer and the applications of game theory, optimization theory, as well as varational inequality theory in these areas. His research is supported by International Postgraduate Research Scholarship (IPRS), Australian Postgraduate Award (APA), and Norman I Price Supplementary Scholarship.
\end{IEEEbiography}
\begin{IEEEbiography}[{\includegraphics[width=1in,height=1.25in,clip,keepaspectratio]
{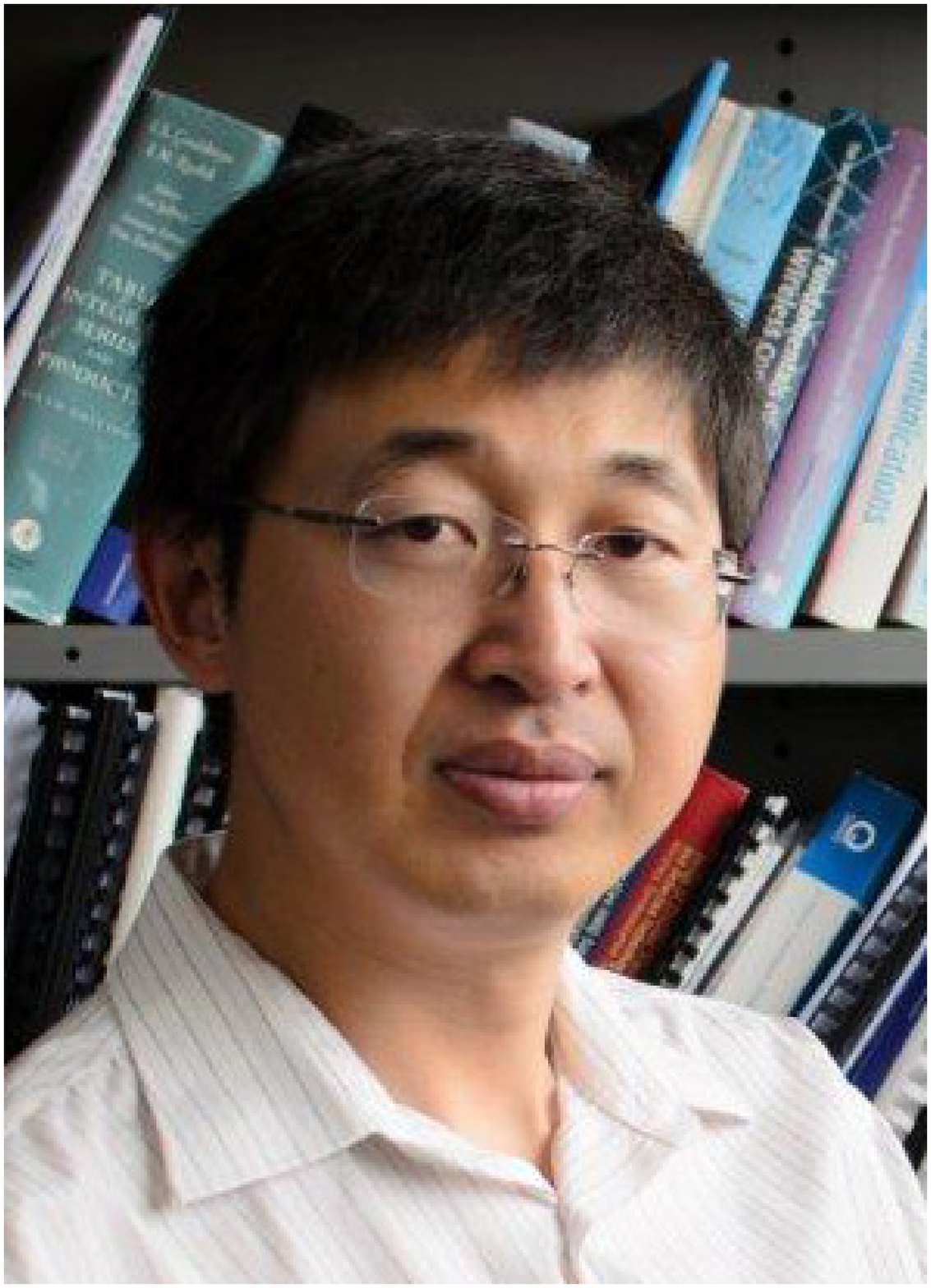}}]{Yonghui Li}
(M'04-SM'09) received his Ph.D. degree in November 2002 from Beijing University of Aeronautics and Astronautics. From 1999-2003, he was affiliated with Linkair Communication Inc, where he held a position of project manager with responsibility for the design of physical layer solutions for the LAS-CDMA system. Since 2003, he has been with the Centre of Excellence in Telecommunications, the University of Sydney, Australia. He is now an Associate Professor in School of Electrical and Information Engineering, University of Sydney. He was the Australian Queen Elizabeth II Fellow and is currently the Australian Future Fellow.

His current research interests are in the area of wireless communications, with a particular focus on machine-to-machine communications, cooperative communications, coding techniques and wireless sensor networks. He holds a number of patents granted and pending in these fields. He is an executive editor for European Transactions on Telecommunications (ETT). He has also been involved in the technical committee of several international conferences, such as ICC, Globecom, etc.
\end{IEEEbiography}
\begin{IEEEbiography}[{\includegraphics[width=1in,height=1.25in,clip,keepaspectratio]
{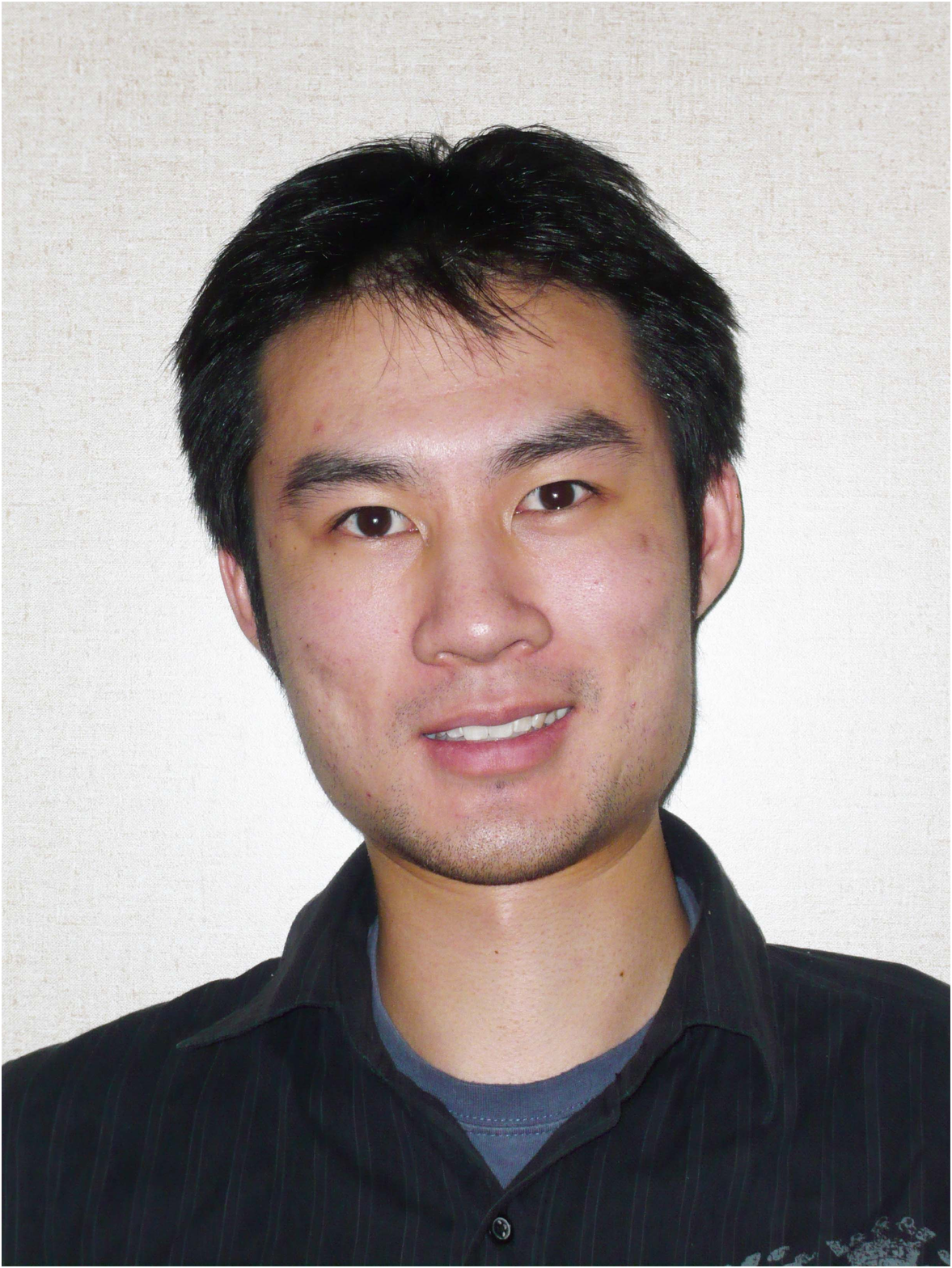}}]{Raymond H. Y. Louie}
(S'06-M'10) received the combined B.E. degree in electrical engineering and
B.Sc. degree in computer science from the University of New South Wales, Sydney, Australia, in 2006 and the Ph.D. degree in electrical engineering from the University of Sydney, Australia, in 2010. He was then an ARC Australian Postdoctoral Fellow at the University of Sydney, and is now a Visiting Assistant Professor at the Hong Kong University of Science and Technology. His research interests include biomedical engineering, cognitive radio, ad hoc networks, MIMO systems, cooperative communications, network coding, and multivariate statistical theory. Dr. Louie was awarded a Best Paper Award at IEEE Globecom 2010.
\end{IEEEbiography}
\begin{IEEEbiography}[{\includegraphics[width=1in,height=1.25in,clip,keepaspectratio]
{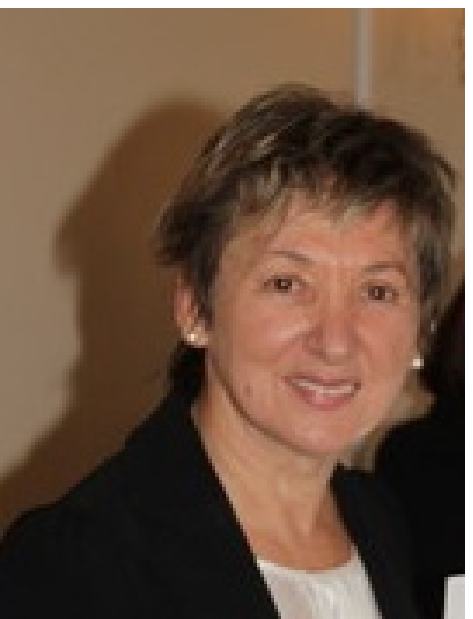}}]{Branka Vucetic}
(SM'00-F'03) currently holds the Peter Nicol Russel Chair of Telecommunications Engineering at the University of Sydney and serves the Director of Centre of Excellence in Telecommunications. She is an internationally recognized expert in wireless communications and coding. She has published more than three hundred research papers and co-authored four books in telecommunications and coding theory.

Prof. Vucetic is an IEEE Fellow. Her most significant research contributions have been in the field of channel coding and its applications in wireless communications. The research of Prof. Vucetic has involved collaborations with industry and government organisations in Australia and several other countries.
\end{IEEEbiography}
\end{document}